\documentclass[11pt]{article}
\usepackage{amssymb,latexsym,amsmath,amsbsy,amsthm,amsxtra,amsgen,graphicx,makeidx,dsfont,xcolor,bbm,amsfonts,tikz}
\usepackage[colorlinks]{hyperref}
\newfont{\msbm}{msbm10 at 11pt}
\newcommand {\R} {\mbox{\msbm R}}

\newcommand {\N} {\mbox{\msbm N}}
\newcommand {\1} {\mathds{1}}

\newtheorem{theorem}{Theorem}[section]
\newtheorem{lemma}[theorem]{Lemma}

\newtheorem{corollary}[theorem]{Corollary}
\newtheorem{proposition}[theorem]{Proposition}

\newtheorem{definition}[theorem]{Definition}
\newtheorem{remark}[theorem]{Remark}

\def\eps{\varepsilon}
\def\Var{\textup{{\bf Var}}}
\DeclareMathOperator{\E}{\mathbf{E}}
\renewcommand{\P}{\mathbf{P}}
\newcommand{\dd}{\mathrm{d}}
\newcommand{\ind}[1]{\1_{\left\{#1\right\}}}
\newcommand{\indset}[1]{\mathbbm{1}_{#1}}
\renewcommand{\bar}[1]{\overline{#1}}
\renewcommand{\hat}[1]{\widehat{#1}}
\renewcommand{\epsilon}{\varepsilon}
\renewcommand{\epsilon}{\varepsilon}
\newcommand{\BBMabs}{BBM${}^+(-\rho)$}
\newcommand{\BBM}{BBM${}(-\rho)$}

\newcommand{\footremember}[2]{
    \footnote{#2}
    \newcounter{#1}
    \setcounter{#1}{\value{footnote}}
}

\begin{document}
\title{The Yaglom limit for branching Brownian motion with absorption and slightly subcritical drift}
\author{Julien Berestycki\footremember{JB}{University of Oxford, E-mail: julien.berestycki@stats.ox.ac.uk}
\and Jiaqi Liu\footremember{JL}{University of Pennsylvania, E-mail: liujiaqi@sas.upenn.edu}
\and Bastien Mallein\footremember{BM}{Institut de Mathématiques de Toulouse, UMR 5219, Université de Toulouse, UPS, F-31062 Toulouse Cedex 9, France. Email: {bastien.mallein@math.univ-toulouse.fr}}
\and Jason Schweinsberg\footremember{JS}{University of California San Diego, Email: jschweinsberg@ucsd.edu}
}
\maketitle

\begin{abstract}
Consider branching Brownian motion with absorption in which particles move independently as one-dimensional Brownian motions with drift $-\rho$, each particle splits into two particles at rate one, and particles are killed when they reach the origin.  Kesten \cite{kesten} showed that this process dies out with probability one if and only if $\rho \geq \sqrt{2}$.  We show that in the subcritical case when $\rho > \sqrt{2}$, the law of the process conditioned on survival until time $t$ converges as $t \rightarrow \infty$ to a quasi-stationary distribution, which we call the Yaglom limit.  We give a construction of this quasi-stationary distribution.
We also study the asymptotic behavior as $\rho \downarrow \sqrt{2}$ of this quasi-stationary distribution.  We show that the logarithm of the number of particles and the location of the highest particle are of order $\epsilon^{-1/3}$, and we obtain a limit result for the empirical distribution of the particle locations.
\end{abstract}


{\small MSC: Primary 60J80; Secondary: 60J65, 60J25

Keywords: Branching Brownian motion, Yaglom limit, quasi-stationary distribution}

\section{Introduction}

A binary branching Brownian motion is a particle system on the real line in which particles move according to independent Brownian motions, while splitting at rate $1$ into two children. This particle system is an archetypal example of a spatial branching process, as the behavior of each particle is independent of its past and of all the other particles alive at the same time.

We consider in this article a branching Brownian motion in $\R_+$ with drift $-\rho$ and absorption at $0$. In this process, each particle moves according to an independent Brownian motion with drift $-\rho$, starting from its current position. If at some time a particle hits the level $0$, it is killed and removed from the process. Additionally, each particle is associated with an independent exponential random time of parameter $1$.  At that time, if the particle is still alive and at position $x > 0$, it is killed and replace by two new particles starting from $x$. We call this process \BBMabs{}.

To construct the \BBMabs{}, we begin by constructing branching Brownian motion without absorption in which particles have drift $-\rho$.  For all $t \geq 0$, we denote by $\mathcal{N}_t$ the set of particles alive at time $t$.  For all $u \in \mathcal{N}_t$ and $s \leq t$, we write $X_s(u)$ for the position at time $s$ of the particle $u$, or its ancestor that was alive at that time.
For all $t \geq 0$, we define
\[
  \mathcal{N}_t^+ := \{ u \in \mathcal{N}_t : \inf_{s \leq t} X_s(u) > 0 \},
\]
which is the set of individuals that stayed in the positive real line up to time $t$.
The state of \BBMabs{} at time $t$ is then encoded by the point measure
\[
  {\bf X}_t := \sum_{u \in \mathcal{N}_t^+} \delta_{X_t(u)},
\]
where $\delta_x$ denotes the unit mass at position $x$. The process $({\bf X}_t, t \geq 0)$ is a Markov process on the set $\mathfrak{P}$ of finite point measures on $\R_+ = (0,\infty)$.

For each probability distribution $\mathcal{D}$ on $\mathfrak{P}$, we denote by $\P^\rho_\mathcal{D}$ the law of the \BBMabs{} starting with an initial condition of law $\mathcal{D}$.  For all $\nu \in \mathfrak{P}$, we write $\P_{\nu}^{\rho}$ for the law of \BBMabs{} started from $\nu$ (i.e. with initial distribution $\delta_{\nu}$).  We also write $\P_x^{\rho}$ for the law of \BBMabs{} started from one particle at $x$ (i.e. with initial distribution $\delta_{\delta_x}$).  We denote by $\E^\rho_\mathcal{D}$, $\E^{\rho}_{\nu}$, and $\E^\rho_x$ the corresponding expectations.

Branching Brownian motion with absorption was first studied in 1978 by Kesten \cite{kesten}, who showed that starting from one particle at $x > 0$, the process dies out almost surely if $\rho \geq \sqrt{2}$ and survives forever with positive probability if $\rho < \sqrt{2}$.  We say that the drift is critical if $\rho = \sqrt{2}$, subcritical if $\rho > \sqrt{2}$, and supercritical if $\rho < \sqrt{2}$. We denote the extinction time for the process by
\begin{equation}
  \label{eqn:defExtinction}
  \zeta := \inf\{ t \geq 0 : \mathcal{N}^+_t = \emptyset \} = \inf\{t \geq 0 : {\bf X}_t(\R^+) = 0\}.
\end{equation}

The asymptotic behaviour of the right tail of the distribution of $\zeta$ was studied by Harris and Harris \cite{hh07}, who showed that when $\rho > \sqrt{2}$, there is a positive constant $K_{\rho}$, depending on $\rho$ but not on $x$, such that
\begin{equation}\label{hhsurvival}
\P_x^{\rho}(\zeta > t) \sim \frac{K_{\rho}}{\sqrt{2 \pi} t^{3/2}} \, x e^{\rho x + (1 - \rho^2/2) t},
\end{equation}
where $\sim$ means that the ratio of the two sides tends to one as $t \rightarrow \infty$.  As a consequence, they deduced that
\begin{equation}
  \label{eqn:meanSize}
  \lim_{t \to \infty} \E^{\rho}_x\left({\bf X}_t(\R_+) \,\middle|\, \zeta > t\right) = \frac{2}{\rho^2 K_\rho}.
\end{equation}
Harris and Harris \cite{hh07} also observed that, as a consequence of arguments due to Chauvin and Rouault \cite{cr88}, there is a probability distribution $\pi^\rho$ on $\N$ with mean $\frac{2}{\rho^2 K_\rho}$ such that for all positive integers $j$, we have
\begin{equation}
  \label{eqn:defPiRho}
  \lim_{t \to \infty} \P^{\rho}_x\left( {\bf X}_t(\R_+) = j \,\middle|\, \zeta > t \right) = \pi^{\rho}(j).
\end{equation}
Our aim in this article to improve on this description of the law of ${\bf X}_t$ conditionally on $\zeta > t$, which we refer to as the Yaglom limit of the \BBMabs{}.  We focus on understanding the asymptotic behavior of this law as $\rho \downarrow \sqrt{2}$. Thus, the present work can be seen as a companion paper to \cite{EnCours} in which we study the asymptotic behavior of $\P_x(\zeta > t)$ as $\rho \downarrow \sqrt{2}$ and $t,x \to \infty$.

\subsection{The Yaglom limit of branching Brownian motion with absorption}

We first study minimal quasi-stationary distributions of the \BBMabs{}. We prove the existence of a unique such distribution as the Yaglom limit of the process. Let us start by defining these different terms.

Let $X = (X_t, t \geq 0)$ be a Feller Markov process on a metric state space $E \cup \{\partial\}$, with $\partial$ an absorbing state.  Let $T = \inf\{t > 0 : X_t = \partial\}$ be the first exit time of $E$ by $X$. We work here under the assumption that $E$ is an irreducible state space for $X$, meaning that for all $x, y \in E$ and $r > 0$, there exists $t > 0$ such that $\P_x(X_t \in B(y,r), T > t) > 0$, where $B(y,r)$ denotes the ball centered at $y$ with radius $r$.

In the context of this article, the Yaglom limit of $X$ is defined as the limit in distribution of the law of $X_t$ conditionally on $T > t$, provided this limit exists. In other words, we say that the probability distribution $\nu$ on $E$ is the Yaglom limit of $X$ if for all continuous bounded functions $f$ and $x \in E$, we have
\begin{equation}
  \label{eqn:defYaglomLimit}
  \lim_{t \to \infty} \E_x(f(X_t) | T > t) = \int_E f \dd \nu.
\end{equation}

On the other hand, a quasi-stationary distribution with parameter $\theta>0$ is defined as a probability distribution $\pi$ on $E$, such that if $X_0$ is started according to the law $\pi$, then the law of $X_t$ conditionally on $T > t$ is $\pi$, and
\begin{equation}\label{theta}
\P_\pi(T > t) = e^{-\theta t}.
\end{equation}
In other words, $\pi$ is a quasi-stationary distribution with parameter $\theta$ of $X$ if for all continuous bounded functions $f$ and $t > 0$, we have
\begin{equation}
  \label{eqn:defQuasiStationary}
  \E_\pi(f(X_t)\ind{T > t}) = e^{-\theta t}\int_E f \dd \pi.
\end{equation}
According to Theorem 2.2 of \cite{cmsm13}, every quasi-stationary distribution satisfies \eqref{theta} for some $\theta > 0$.
A \emph{minimal quasi-stationary distribution} of $X$ is defined, when it exists, as a quasi-stationary distribution with parameter $\mu$ such that no quasi-stationary distribution with parameter larger than $\mu$ exists.

We show that $({\bf X}_t, t\geq 0)$ admits a Yaglom limit, which is a minimal quasi-stationary distribution of the process. The state space we are working with is $\mathfrak{P} = \mathfrak{P}^* \cup \{0\}$, the set of finite point measures on $\R_+$, with $0$ representing the null point measure and $\mathfrak{P}^*$ consisting of all nonzero point measures. We remark that $0$ is an absorbing state for $({\bf X}_t, t\geq 0)$.
We endow the space $\mathfrak{P}$ with the topology of weak convergence, which means $\mu_n \to \mu$ in $\mathfrak{P}$ as $n \to \infty$ if $\int f \: d\mu_n \rightarrow \int f \: d\mu$
for all continuous bounded functions $f$. This topology makes $\mathfrak{P}$ a metrizable complete space.
The following result proves the existence of the Yaglom limit of $({\bf X}_t, t \geq 0)$.

\begin{proposition}
\label{prop:existenceYaglom}
For all $\rho > \sqrt{2}$, there exists a probability distribution $\mathcal{D}^\rho$ on $\mathfrak{P}^*$ such that for all continuous bounded functions $F : \mathfrak{P}^* \to \R$ and $\nu \in \mathfrak{P}^*$, we have
\begin{equation}
  \label{eqn:yaglomPointMeasure}
  \int F({\bf x})\mathcal{D}^\rho(\dd {\bf x}) = \lim_{t \to \infty} \E^\rho_\nu( F({\bf X}_t) \,|\, \zeta > t ).
\end{equation}
\end{proposition}

The measure $\mathcal{D}^\rho$ is defined in Section~\ref{sec:spine} through a backward spinal construction of the \BBMabs{} seen from its right-most particle. This backward construction gives us a probabilistic interpretation of the constant $K_\rho$ appearing in \eqref{hhsurvival}, see Remark~\ref{rem:krho}. In particular, this allows the computation of $K_\rho$ by Monte Carlo simulations.

We next show that the Yaglom law $\mathcal{D}^\rho$ is also a minimal quasi-stationary distribution of \BBMabs{}. Similar results are known to hold for multiple examples of Markov chains, see \cite{cmsm13} or \cite{MV12} among others.

\begin{corollary}
\label{survivalcor}
Let $\rho > \sqrt{2}$. For all continuous bounded functions $F : \mathfrak{P}^* \to \R$, we have
\[
  \E_{\mathcal{D}^\rho}^{\rho} \left( F (\mathbf{X}_t) \ind{\zeta > t} \right) = e^{-(\rho^2/2 - 1)t} \int F(\bf x)\mathcal{D}^\rho(\dd \bf x).
\]
In particular, we have $$\P^{\rho}_{\mathcal{D}^{\rho}}(\zeta > t) = e^{-(\rho^2/2 - 1)t}$$ and
\[
 \E_{\mathcal{D}^\rho}^{\rho} \left( F(\mathbf{X}_t) \, | \, \zeta>t \right) =  \E_{\mathcal{D}^\rho}^{\rho} (F(\mathbf{X}_0)).
\]
Moreover, there exists no quasi-stationary distribution for \BBMabs{} with a larger parameter than $\rho^2/2 - 1$.
\end{corollary}

\subsection{Properties of the Yaglom limit as $\eps \rightarrow 0$}

Let $\nu$ be an arbitrary configuration in $\mathfrak{P}$ described by
\begin{equation}\label{nudef}
\nu = \sum_{i \in I} \delta_{x_i}
\end{equation}
of particles at the locations $(x_i, i \in I)$, where $I$ is a finite set of indices. We denote by $$N(\nu) = \nu(\R^+)= \# I$$ the number of particles in the configuration $\nu$, and by $$M(\nu) = \max_{i \in I} x_i = \sup\{a \geq 0: \nu([a, \infty)) > 0\}$$ the location of the right-most particle in $\nu$.

We describe here some of the properties of the quasi-stationary measure when $\rho > \sqrt{2}$, related to the number and configuration of particles.  These results can equivalently be interpreted as results for the behavior of the \BBMabs{} conditioned to survive for a long time. Preliminary results in this direction were obtained by Harris and Harris \cite{hh07}. Observe that \eqref{eqn:defPiRho} can be seen as a preliminary to Proposition~\ref{prop:existenceYaglom}, and $\pi^\rho$ is the image measure of $\mathcal{D}^\rho$ by $\nu \mapsto N(\nu)$. More precisely, let $D^{\rho}$ denote a point measure having law $\mathcal{D}^\rho$, we have
\[
  \P(D^\rho(\R_+) = j) = \pi^{\rho}(j) \quad \text{and} \quad \E(D^\rho(\R_+)) = \frac{2}{\rho^2 K_\rho},
\]
using \cite[Theorem 3]{cr88} to identify the first moment of $\pi^\rho$. These equalities can be compared to \eqref{eqn:meanSize} and \eqref{eqn:defPiRho}.

The primary aim of the present article is to study the asymptotic behavior of the Yaglom measure $\mathcal{D}^\rho$ as $\rho \to \sqrt{2}$.  That is, we consider the case of slightly subcritical drift, so that $$\rho = \sqrt{2} + \eps$$ with $\eps>0$ and $\eps \to 0$. Liu \cite{liu21} obtained upper and lower bounds for $K_{\rho}$ as $\rho \downarrow \sqrt{2}$, showing that there exist two constants $0 < C_1 < C_2 < \infty$ such that for $\eps > 0$ sufficiently small, we have
\begin{equation}\label{meanN}
e^{C_1/\sqrt{\eps}} \leq \E[D^{\rho}(\R_+)] \leq e^{C_2/\sqrt{\eps}}.
\end{equation}

We obtain here a much more precise result for the configuration $D^{\rho}$ of particles under the quasi-stationary distribution as $\rho \to \sqrt{2}$.
We show that $M(D^\rho)$ and $\log N(D^\rho)$ both are of the order $\epsilon^{-1/3}$ as $\epsilon = \rho - \sqrt{2} \to 0$. More precisely, we have the following result.

\begin{theorem}\label{NandM}
Let $V$ have an exponential distribution with mean one.  For BBM${}^+(-\rho$) started from the quasi-stationary distribution $\mathcal{D}^{\rho}$, as $\eps = \rho - \sqrt{2} \rightarrow 0$ we have the joint convergence in distribution
\[
  \left( \eps \zeta, \eps^{1/3} \log N({\bf X}_0), \eps^{1/3} M({\bf X}_0) \right) \Rightarrow \left( \frac{1}{\sqrt{2}} V, \frac{(3 \pi^2)^{1/3}}{2^{1/6}} V^{1/3}, \frac{(3 \pi^2)^{1/3}}{2^{2/3}} V^{1/3} \right).
\]
\end{theorem}

This result is quite similar to Theorem 1.5 in \cite{ms22}, which applies to branching Brownian motion with absorption in the case of critical drift $\rho = \sqrt{2}$, conditioned to survive for an unusually long time.  Note that, because ${\bf X}_0$ has the same law as $D^{\rho}$ by definition, the convergence of the second and third coordinates can equivalently be written as $$(\eps^{1/3} \log N(D^{\rho}), \eps^{1/3} M(D^{\rho})) \Rightarrow ((3 \pi^2)^{1/3} 2^{-1/6} V^{1/3}, (3 \pi^2)^{1/3}2^{-2/3} V^{1/3}).$$

Corollary \ref{survivalcor} implies that under $\mathcal{D}^{\rho}$, the amount of time for which {BBM$^{+}(-\rho)$} survives is exponentially distributed with rate $(\rho^2/2) - 1$, which implies the convergence $\eps \zeta \Rightarrow V/\sqrt{2}$ as $\eps \rightarrow 0$.  Because the survival time is comparable to $\eps^{-1}$, particles drift to the left only a constant distance more over this time period than they would in the case of critical drift.  Therefore, we can apply results from \cite{ms22} for the case of critical drift to study {BBM$^{+}(-\rho)$} over this time period.
The key idea, which is also behind the results in \cite{ms22}, is that when the particles are in the quasi-stationary distribution, for the process to survive for an additional time $t$, the location of the right-most particle will be close to $ct^{1/3}$ and the logarithm of the number of particles will be close to $\sqrt{2} ct^{1/3}$, where $c = (3 \pi^2)^{1/3}/\sqrt{2}.$
In particular, if the process will survive for an additional time $V/(\sqrt{2} \eps)$, then the location of the right-most particle will be close to $2^{-1/6} c V^{1/3} \eps^{-1/3}$, and the logarithm of the number of particles will be close to $2^{1/3} c V^{1/3} \eps^{-1/3}$, which is consistent with Theorem \ref{NandM}.

One aspect of this result which may at first be surprising is that the logarithm of the number of particles is of the order $
\eps^{-1/3}$, while (\ref{meanN}) implies that the logarithm of the expected number of particles is comparable to $\eps^{-1/2}$.  The reason for this discrepancy is that the mean is dominated by rare events in which the number of particles is unusually large.  Indeed, if the logarithm of the number of particles were exactly $2^{1/3} c V^{1/3} \eps^{-1/3}$, then the expected number of particles would be
$$\int_0^{\infty} e^{2^{1/3} c s^{1/3} \eps^{-1/3}} \cdot e^{-s} \: ds,$$
and the integrand is maximized when $s$ is of the order $\eps^{-1/2}$, consistent with \eqref{meanN}.

We can also describe the empirical measure of the \BBMabs{} under its quasi-stationary distribution as $\rho \to \sqrt{2}$. The limit is the same as the one obtained in \cite{ms22} for the Yaglom limit of the BBM$^+(-\sqrt{2})$.  For $\nu \in \mathfrak{P}$ defined as in \eqref{nudef}, we define
$$\chi(\nu) = \frac{1}{N(\nu)} \sum_i \delta_{x_i} = \frac{\nu}{\nu(\R^+)}$$ and
$$\eta(\nu) = \bigg( \sum_i e^{\sqrt{2} x_i} \bigg)^{-1} \sum_i e^{\sqrt{2} x_i} \delta_{x_i/M(\nu)}.$$
We note that $\xi(\nu)$ is the empirical distribution of $\nu$ and $\eta(\nu)$ characterizes the empirical distribution of particles close to the right-most position in $\nu$. The following theorem shows that as the process approaches criticality from the subcritical case, under the Yaglom limit law, the empirical distribution has a density proportional to $ye^{-\sqrt{2}y}$ and the empirical distribution of particles close to the right-most position has a sinusoidal shape.

\begin{theorem}\label{etachi}
Let $\mu$ be the probability measure on $(0, \infty)$ with density $h_1(y) = 2ye^{-\sqrt{2} y}$, and let $\xi$ be the probability measure on $(0,1)$ with density $h_2(y) = \frac{\pi}{2} \sin(\pi y)$.  Then, as $\eps \rightarrow 0$, we have
$$\chi(D^{\rho}) \Rightarrow \mu \qquad\mbox{and}\qquad\eta(D^{\rho}) \Rightarrow \xi,$$ where $\Rightarrow$ denotes convergence in distribution for random elements in the Polish space of probability measures on $(0, \infty)$, endowed with the weak topology.
\end{theorem}

The rest of the article is organized as follows. We construct the Yaglom limit of \BBMabs{} in Section~\ref{sec:spine}, showing in particular Proposition~\ref{prop:existenceYaglom} and Corollary \ref{survivalcor}. We then prove Theorems~\ref{NandM} and \ref{etachi} in Section~\ref{Yaglomsec}, largely by adapting the arguments of Maillard and Schweinsberg \cite{ms22}.

\section{Construction of the Yaglom limit}\label{sec:spine}

We study the asymptotic behaviour of \BBMabs{} conditionally on surviving for a long time using the same method as the one used in \cite{bbcm} to study the large deviations of branching Brownian motion or \cite{ll} for the large deviations of branching random walks.

We first introduce the spinal decomposition of the branching Brownian motion without absorption. This technique allows for the representation of the law of branching Brownian motion biased by an additive martingale as a branching process with a distinguished particle called the spine.  By conditioning the spine to be the right-most particle in the process, we obtain a representation of the Yaglom limit law $\mathcal{D}^\rho$.

Recall that for all $t > 0$, we denote by $\mathcal{N}_t$ the set of particles alive at time $t$ in a binary branching Brownian motion with drift $-\rho$, which we call \BBM{},
while $\mathcal{N}_t^+$ denotes the set of particles that stay in the positive half-line until time $t$.
For $t \geq 0$, we write $\mathcal{F}_t = \sigma(X_s(u), u \in \mathcal{N}_t, s \leq t)$ for the natural filtration associated with the branching Brownian motion. We define the additive martingale of this process as
\[
  W_t := \sum_{u \in \mathcal{N}_t} e^{\rho X_t(u) - (1 - \rho^2/2)t}.
\]
Using that $(W_t, t \geq 0)$ is a non-negative martingale, the biased law of the branching Brownian motion is defined as
\[
  \forall t \geq 0, \forall A \in \mathcal{F}_t, \quad \bar{\P}^\rho_x(A) = e^{-\rho x}\E^\rho_x(W_t \indset{A}).
\]
This law can be represented as a branching Brownian motion with spine. This representation was introduced by Chauvin and Rouault \cite{cr88}, extended to branching random walks by Lyons \cite{Lyo97}, and finally extended by Biggins and Kyprianou \cite{BiK04} to general branching processes.

The \BBM{} with spine is defined as follows. A spine particle, starting from position $x$ at time $0$, moves according to a Brownian motion without drift, while giving birth at rate $2$ to new particles. Each of these newborn particles then starts an independent copy of the \BBM. At time $t$, we denote by $w_t$ the label of the spine particle. We denote by $\hat{\P}_x$ the law of this process, and $(\mathcal{F}_t, t \geq 0)$ the filtration related to the positions of particles in the process with spine, without the information on the spine. The so-called spine decomposition of the branching Brownian motion (see \cite[Theorem~5]{cr88}, \cite{Lyo97} or \cite[Theorem 8.1]{hh04} for iterations on this argument) is the following result.

\begin{proposition}
\label{prop:spineDecomposition}
For all $A \in \mathcal{F}_t$, we have $\bar{\P}^\rho_x(A) = \hat{\P}^\rho_x(A)$. Moreover, we have
\[
  \hat{\P}^\rho_x(w_t = u| \mathcal{F}_t) = \frac{e^{\rho X_t(u)-(1 - \rho^2/2)t}}{W_t}.
\]
\end{proposition}

Using the spine decomposition, we now propose a description of the law of $(\mathbf{X}_t, t\geq 0)$ through a point measure that we now describe. Let $\beta^{z,x,t}$ be a three-dimensional Bessel bridge of length $t$ from $z$ to $x$, and let $\tau = (\tau_j, j \geq 1)$ be the atoms of an independent Poisson point process with intensity $2$. Conditionally on $(\beta^{z,x,t},\tau)$, for each positive integer $j$, let $(\tilde{\mathbf{X}}^{(j)}_{s}, s \geq 0)$ be an independent copy of \BBMabs{} started from a single particle at position $\beta^{z,x,t}_{\tau_j}$ at time $0$. We define
\begin{equation}
  \label{eqn:defBarD}
  \tilde{\mathbf{X}}^{x,z}_t = \delta_z + \sum_{j=1}^\infty \ind{\tau_j < t} \tilde{\mathbf{X}}^{(j)}_{\tau_j}.
\end{equation}
In other words, to construct $\tilde{\mathbf{X}}^{x,z}_t$, we view $\beta^{z,x,t}$ as a Bessel bridge moving backward in time from $z$ at time $0$ to $x$ at time $-t$, while giving birth at rate $2$ to particles. These particles then start independent copies (forward in time) of a \BBMabs{}. The point measure $\tilde{\mathbf{X}}^{x,z}_t$ then describes the set of particles alive at time $0$.

\begin{figure}[ht]
\centering
\input{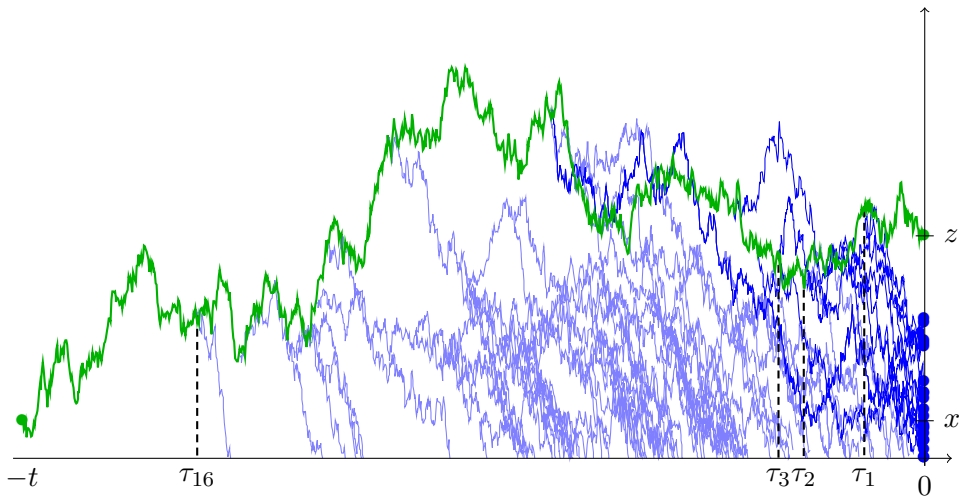}
\caption{Construction of the auxiliary point measure $\bar{D}^{\rho,x,z}_t$. The trajectory of the Bessel bridge is drawn in green, while the \BBMabs{} spawning from this spine are drawn in blue. The point measure is the set of point masses on $\R_+$ drawn at time $0$.}
\end{figure}

The following result allows one to describe the law of \BBMabs{} at time $t$ in terms of $\tilde{\mathbf{X}}^{x,z}_t$. The backward construction of $\tilde{\mathbf{X}}^{x,z}_t$ will make it convenient to consider limiting behavior as $t \to \infty$.

\begin{lemma}\label{EFDtLem}
Let $F$ be a measurable bounded function on $\mathfrak{P}$ such that $F(0) = 0$. For all $t \geq 0$ and $x > 0$, we have
\begin{equation*}
  e^{- t(1 - \rho^2/2)} \E^\rho_x\left( F(\mathbf{X}_t) \right)
   =  \int_0^\infty e^{\rho(x-z)}e^{-\frac{(x-z)^2}{2t}} \frac{1 - e^{-2xz/t}}{\sqrt{2 \pi t}} \E\left( F(\tilde{\mathbf{X}}^{x,z}_t) \ind{\tilde{\mathbf{X}}^{x,z}_t((z,\infty)) = 0} \right) \dd z.
\end{equation*}
\end{lemma}

\begin{proof}
Given $t > 0$, we write $m^+_t \in \mathcal{N}^+_t$ for the label of the (a.s. unique) particle at the largest position at time $t$ among the particles that survived the absorption. If $\mathcal{N}^+_t = \emptyset$, we give an arbitrary value to $m^+_t$, and observe that on that event, $F(\mathbf{X}_t) = 0$. We observe that we can decompose $\E^\rho_x(F(\mathbf{X}_t))$ according to the label of the particle reaching the maximal position:
\begin{align*}
\E^\rho_x \left( F(\mathbf{X}_t) \right) &= \E^\rho_x\left(  F(\mathbf{X}_t) \sum_{u \in \mathcal{N}^+_t} \ind{m^+_t = u} \right)\\
&= \E^\rho_x\left( \sum_{u \in \mathcal{N}_t}  F(\mathbf{X}_t) \ind{m^+_t = u} \right) \\
&= \E^\rho_x\left( \sum_{u \in \mathcal{N}_t} F(\mathbf{X}_t) \ind{{X}_s(u) > 0 \: \forall s \leq t} \ind{\mathbf{X}_t(({X}_t(u),\infty)) = 0} \right).
\end{align*}
Therefore, using Proposition \ref{prop:spineDecomposition},
\begin{align*}
  &\E^\rho_x\left( F(\mathbf{X}_t) \right)\\
  &\qquad = \bar{\E}^\rho_x \left( \frac{e^{\rho x}}{W_t} \sum_{u \in \mathcal{N}_t} F(\mathbf{X}_t) \ind{X_s(u) > 0 \: \forall s \leq t} \ind{\mathbf{X}_t((X_t(u),\infty)) = 0} \right)\\
  &\qquad = \hat{\E}^\rho_x \left( \sum_{u \in \mathcal{N}_t} e^{\rho (x-X_t(u)) + (1 - \rho^2/2)t} \hat{\P}^\rho_x (w_t= u|\mathcal{F}_t) F(\mathbf{X}_t) \ind{X_s(u) > 0 \: \forall s \leq t} \ind{\mathbf{X}_t((X_t(u),\infty)) = 0} \right)\\
  &\qquad = e^{(1 - \rho^2/2)t} \hat{\E}^\rho_x \left( e^{\rho (x-X_t(w_t))}  F(\mathbf{X}_t) \ind{X_s(w_t) > 0 \: \forall s \leq t} \ind{\mathbf{X}_t((X_t(w_t),\infty)) = 0} \right).
\end{align*}
Because a three-dimensional Bessel bridge of length $t$ from $z$ to $x$ has the same law as the Brownian bridge of length $t$ from $z$ to $x$ conditioned to stay positive, $\tilde{\mathbf{X}}^{x,z}_t$ has the same law as ${\bf X}_t$ conditioned on $w_t = z$ and $X_s(w_t) > 0$ for all $s \leq t$.  As a result, conditioning on the value of $X_t(w_t)$ and then using the backward construction of the branching Brownian motion with spine, we have
\begin{align*}
&e^{-(1 - \rho^2/2)t} \E^\rho_x\left( F(\mathbf{X}_t) \right)\\
  &\quad = \int_0^\infty e^{\rho(x-z)} \frac{e^{- \frac{(x-z)^2}{2t}}}{\sqrt{2\pi t}} \P_x(B_s > 0, s\leq t|B_t = z) \E\left(F(\tilde{\mathbf{X}}^{x,z}_t) \ind{\tilde{\mathbf{X}}^{x,z}_t((z,\infty)) = 0} \right) \dd z,
\end{align*}
where $B$ is a Brownian motion started from $x$ under $\P_x$.
Finally, we have $$\P_x(B_s > 0, s \leq t | B_t = z) = 1 - e^{ - 2 x z/t}$$ using the reflection principle (see e.g. Proposition 3 of \cite{Sch92}), which completes the proof.
\end{proof}

Next, we consider the limit of $\tilde{\bf X}_t^{x,z}$ as $t\to\infty$. Let $R$ be a Bessel process of dimension 3 started from $z$ and $(\tau_j)_{j=1}^\infty$ be the atoms of an independent Poisson point process with intensity $2$. Conditionally on $(R,(\tau_j)_{j=1}^\infty)$, let $(\check{\mathbf{X}}^{(j)}_t, t\geq 0)_{j=1}^\infty$ be independent BBM$^+(-\rho$) processes, such that $(\check{\mathbf{X}}^{(j)}_t, t\geq 0)$ starts with one particle at position $R_{\tau_j}$ at time $0$. We show that $\tilde{\mathbf{X}}^{x,z}_t$ converges in law, as $t \to \infty$ towards the random variable
\[
  \check{\mathbf{X}}^{z} := \delta_z + \sum_{j=1}^\infty \check{\mathbf{X}}^{(j)}_{\tau_j}.
\]

\begin{lemma}\label{lem:convergeofX}
\label{lem}
For all $x > 0$, $\tilde{\mathbf{X}}^{x,z}_t \to \check{\mathbf{X}}^{z}$ in law as $t \rightarrow \infty$, and $\P(\check{\mathbf{X}}^{z}((z,\infty)) = 0)>0$.
\end{lemma}

To prove this lemma, we make use of the following well-known fact, that the overall maximum of the BBM$(-\rho)$ has exponential tails. This result can be found for example in \cite[Equation~(4.12)]{madaule}, in the context of branching random walks. For all $x \geq 0$, we have
\begin{equation}
\label{fct:maxDep}
  \P^0_0\left(\sup_{t \geq 0, u \in \mathcal{N}_t} X_t(u) - \sqrt{2} t \geq x \right) = \P^\rho_0\left( \sup_{t\geq 0, u \in \mathcal{N}_t} X_t(u) - (\sqrt{2} -\rho)t \geq x \right) \leq e^{-\sqrt{2}x}.
\end{equation}
This result can alternatively be proved by computing the expected value of the $\P^0$-martingale $V_t = \sum_{u \in \mathcal{N}_t} e^{\sqrt{2} (X_t(u) - \sqrt{2} t)}$, stopped when a particle first crosses the curve $s \mapsto x + \sqrt{2} s$.

\begin{proof}[Proof of Lemma~\ref{lem}]
Recall that $\beta^{z,x,t}$ denotes a three-dimensional Bessel bridge from $z$ to $x$ of length $t$.
Also, recall that for each fixed $s \geq 0$, we have (see, for example, p. 70 of \cite{hh07})
\[
  \lim_{t \to \infty} (\beta^{z,x,t}_u, 0 \leq u \leq s) = (R^z_u, 0 \leq u \leq s) \text{ in distribution,}
\]
where $R^z$ is a three-dimensional Bessel process starting from $z$.

Next, using \eqref{fct:maxDep} and a trivial coupling, we note that
\[
  \sum_{j \in \mathbb{N}} \P\left(M(\tilde{\mathbf{X}}^{(j)}_{\tau_j}) \geq \beta^{z,x,t}_{\tau_j} + \frac{\sqrt{2} - \rho}{2} \tau_j\right) \leq \E\left( \sum_{j \in \mathbb{N}} e^{-\sqrt{2}(\rho - \sqrt{2})\tau_j/2}\right) < \infty.
\]
Therefore, by the Borel-Cantelli lemma, almost surely there exists $n \in \N$ such that for all $j \geq n$, we have
\[
  M(\mathbf{X}^{(j)}_{\tau_j}) \leq \beta^{z,x,t}_{\tau_j} + \frac{\sqrt{2} - \rho}{2}\tau_j.
\]

Finally, we use the law of iterated logarithm for the Bessel bridge $\beta^{z,x,t}$ and the Bessel process $R^z$, which implies
\[
  \left(\sup_{0  \leq s \leq t} \frac{\beta^{z,x,t}_s}{1+s^{2/3}}, t \geq 0 \right) \text{ is tight, and } \sup_{t \geq 0} \frac{R^z_t}{1+t^{2/3}} < \infty.
\]
Hence, for all $\eta > 0$, there exists $N \in \N$ such that with probability at least $1-\eta$, we have for all $t > 0$,
\[
  \tilde{\mathbf{X}}^{x,z}_t = \delta_z + \sum_{j=1}^N \ind{\tau_j < t} \tilde{\mathbf{X}}^{(j)}_{\tau_j} \text{ and similarly } \check{\mathbf{X}}^{z} = \delta_z + \sum_{j=1}^N \check{\mathbf{X}}^{(j)}_{\tau_j}.
\]
We complete the proof by letting $t$ and then $N$ grow to $\infty$, and using that $(\beta^{z,x,t}_s, s \leq \tau_N) \to (R^z_s, s \leq \tau_N)$ in law as $t \to \infty$.  Therefore, in particular $(\tilde{\mathbf{X}}^{(j)}, j \leq N) \to (\check{\mathbf{X}}^{(j)}, j \leq N)$ in law, and it follows that $\tilde{\mathbf{X}}^{x,z}_t \to \check{\mathbf{X}}^{z}$ in law as $t \rightarrow \infty$, which is the first part of the lemma.  The second claim in the lemma follows because $N < \infty$ and for each $j$, $\check{\mathbf{X}}^{(j)}_{\tau_j}((z, \infty)) > 0$ with positive probability.
\end{proof}

\begin{lemma}
\label{lem:yaglomLimit}
Let $F$ be a continuous bounded function on $\mathfrak{P}$ such that $F(0) = 0$. We have
\begin{equation*}
  \lim_{t \to \infty} t^{3/2} e^{t(\frac{\rho^2}{2}- 1) }\E_x^\rho\left( F(\mathbf{X}_t) \right)
   = x e^{\rho x}\sqrt{\frac{2}{\pi}}\int_0^\infty ze^{-\rho z}  \E\left( F(\check{\mathbf{X}}^{z}) \ind{\check{\mathbf{X}}^{z}((z,\infty)) = 0} \right) \dd z.
\end{equation*}
\end{lemma}

\begin{proof}
For all $z \geq 0$, it follows from Lemma \ref{lem:convergeofX} that
\[
  \lim_{t \to \infty} \E\left(F(\tilde{\mathbf{X}}^{x,z}_t) \ind{\tilde{\mathbf{X}}^{x,z}_t((z,\infty)) = 0} \right) =  \E\left( F(\check{\mathbf{X}}^{z}) \ind{\check{\mathbf{X}}^{z}((z,\infty)) = 0} \right),
\]
while being bounded by $\|F\|_\infty$. Hence, using the dominated convergence theorem, we have
\begin{multline*}
  \lim_{t \to \infty} \int_0^\infty e^{-\rho z} e^{- \frac{(x-z)^2}{2t}}  t\left(1 - e^{-2xz/t} \right) \E\left(F(\tilde{\mathbf{X}}^{x,z}_t) \ind{\tilde{\mathbf{X}}^{x,z}_t((z,\infty)) = 0}  \right) \dd z\\
  = 2x \int_0^\infty z e^{-\rho z} \E\left(F(\check{\mathbf{X}}^{z}) \ind{\check{\mathbf{X}}^{z}((z,\infty)) = 0} \right) \dd z.
\end{multline*}
The proof is now complete through an application of Lemma \ref{EFDtLem}.
\end{proof}

\begin{remark}
\label{rem:krho}
In particular, taking $F$ to be the function $D \mapsto \ind{D \neq 0}$, Lemma~\ref{lem:yaglomLimit} yields the following representation of the constant $K_{\rho}$ from (\ref{hhsurvival}):
\[
  K_{\rho} = 2\int_0^\infty ze^{-\rho z} C_\rho(z) \dd z,
\]
where $C_\rho(z) = \P(\check{\mathbf{X}}^{z}((z,\infty))= 0)$.
\end{remark}

For all measurable subsets $A$ of $\mathfrak{P}$, define
$$\mathcal{D}^{\rho}(A) = \frac{\int_0^\infty ze^{-\rho z}  \P\left( \check{\mathbf{X}}^{z} \in A, \check{\mathbf{X}}^{z}((z,\infty)) = 0 \right) \dd z}{\int_0^\infty ze^{-\rho z}  \P\left( \check{\mathbf{X}}^{z}((z,\infty)) = 0 \right)\dd z}.$$
We observe that $\mathcal{D}^\rho$ is a probability measure on $\mathfrak{P}$.
Corollary \ref{cor:proofOfProp:existenceYaglom} below shows that $\mathcal{D}^{\rho}$, defined in this way, is the
limiting distribution in \eqref{eqn:yaglomPointMeasure}, which proves Proposition~\ref{prop:existenceYaglom} for $\nu = \delta_x$.

\begin{corollary}
\label{cor:proofOfProp:existenceYaglom}
For all continuous bounded functions $F$ with $F(0) = 0$ and all $x \geq 0$, we have
\[
  \lim_{t \to \infty} \E^\rho_x(F(\mathbf{X}_t) | \zeta > t) = \int_{\mathfrak{P}} F(\mathbf{x}) \: \mathcal{D}^{\rho}(\dd \mathbf{x}).
\]
\end{corollary}

\begin{proof}
By a direct application of Lemma~\ref{lem:yaglomLimit}, we have
\begin{align*}
  \lim_{t \to \infty}  \E^\rho_x(F(\mathbf{X}_t) | \zeta > t)
  &= \lim_{t \to \infty} \frac{t^{3/2} e^{-t(1-\rho^2/2)} \E^\rho_x(F(\mathbf{X}_t))}{t^{3/2} e^{-t(1-\rho^2/2)} \P_x^\rho(\mathbf{X}_t \neq 0)} \\
  &= \frac{\int_0^\infty ze^{-\rho z}  \E\left( F(\check{\mathbf{X}}^{z})\ind{\check{\mathbf{X}}^{z}((z,\infty)) = 0} \right) \dd z}{\int_0^\infty ze^{-\rho z}  \P\left( \check{\mathbf{X}}^{z}((z,\infty)) = 0 \right)\dd z} \\
  &= \int_{\mathfrak{P}} F(\mathbf{x}) \: \mathcal{D}^{\rho}(\dd \mathbf{x}),
\end{align*}
as claimed.
\end{proof}

\begin{remark}
\label{rem:yaglomConstruction}
Observe that we can sample from $\mathcal{D}^\rho$ in the  following two equivalent ways:
\begin{itemize}
\item Start by sampling a random variable $Z$ on $\R_+$ with probability distribution proportional to $z e^{-\rho z} C_\rho(z) \dd z$. Conditionally on $Z,$ then sample a copy $D^\rho$ of $\check{\mathbf{X}}^{Z}$ conditioned on $\check{\mathbf{X}}^{Z}((Z,\infty)) = 0$. The point measure $D^\rho$ has law $\mathcal{D}^\rho$. The conditioning does not affect the law of $Z$.
\item Start by sampling a random variable $Z$ according to the probability distribution $\rho^2 z e^{-\rho z} \dd z$ on $\R_+$. The law of $\check{\mathbf{X}}^{Z}$ conditionally on the event $\check{\mathbf{X}}^{Z}((Z,\infty)) = 0$ is $\mathcal{D}^\rho$. The conditioning affects the law of $Z$ (as should be expected).
\end{itemize}
\end{remark}

We are now able to prove Proposition~\ref{prop:existenceYaglom}.

\begin{proof}[Proof of Proposition \ref{prop:existenceYaglom}]
Let $\nu = \sum_{j=1}^n \delta_{x_j} \in \mathfrak{P}$. For all continuous bounded functions $F$, we have
\[
  \E_\nu(F({\bf X}_t) | \zeta > t) = \frac{\E_\nu(F({\bf X}_t)\ind{\zeta> t})}{\P_\nu(\zeta > t)}.
\]
Using the branching property, we observe that we can write
\[
  {\bf X} = \sum_{j=1}^n {\bf X}^{(j)} \quad \text{and} \quad \zeta = \max_{j \leq n} \zeta_j
\]
where the $X^{(j)}$ are independent \BBMabs{} processes with law $\P_{x_j}^\rho$ and $\zeta_j = \inf\{t > 0: {\bf X}^{(j)} = 0\}$ is the extinction time of $X^{(j)}$.

Using \eqref{hhsurvival} and the independence, there exists $C > 0$ (which depends on $\nu$) such that
\[
  \P^\rho_\nu( \#\{j : \zeta_j > t\} \geq 2) \leq C e^{- 2(\rho^2/2 - 1)t}.
\]
It follows that, using $\sim$ to denote that the ratio of the two sides tends to infinity as $t \rightarrow \infty$,
\begin{equation}\label{prop1.1a}
  \P^\rho_\nu(\zeta > t) \sim \sum_{j=1}^n \P^\rho_\nu(\zeta_j > t) = \sum_{j=1}^n \P^\rho_{x_j}(\zeta > t).
\end{equation}
and
\begin{equation}\label{prop1.1b}
\bigg| \E^\rho_\nu(F({\bf X}_t)\ind{\zeta > t}) - \sum_{j=1}^n \E^\rho_\nu(F({\bf X}^{(j)}_t)\ind{\zeta_j > t}) \bigg| \leq C (n+1) e^{-2(\rho^2/2-1)t} \|F\|_\infty.
\end{equation}
Therefore, letting $t \to \infty$ and using Corollary \ref{cor:proofOfProp:existenceYaglom}, we have
\begin{equation}\label{prop1.1c}
  \E_\nu^\rho (F ({\bf X}^{(j)}_t) \ind{\zeta_j > t}) = \E_{x_j}^\rho(F({\bf X}_t) \ind{\zeta > t}) \sim \P_{x_j}^\rho(\zeta > t) \int_{\mathfrak{P}} F({\bf x})\mathcal{D}^\rho(\dd {\bf x}).
\end{equation}
The result now follows from \eqref{prop1.1a}, \eqref{prop1.1b}, and \eqref{prop1.1c}.
\end{proof}

These results are enough to prove that $\mathcal{D}^\rho$ is a minimal quasi-stationary distribution associated to the \BBMabs{}.

\begin{proof}[Proof of Corollary \ref{survivalcor}]
Let $F$ be a continuous bounded function. We remark that, using the branching property, for all $s,t \geq 0$ we have
\[
  \E^\rho_x(F(\mathbf{X}_{t+s}) \ind{\zeta > t+s}) = \E_x^\rho\left( \E^\rho_{\mathbf{X}_t}(F(\mathbf{X}_s) \ind{ \zeta > s}) \ind{\zeta > t}\right).
\]
Hence, using Proposition~\ref{prop:existenceYaglom} with the function $G : \mathbf{x} \mapsto \E^\rho_\mathbf{x}(F(\mathbf{X}_s) \ind{ \zeta > s})$, we have
\[
  \lim_{t \to \infty} \E^\rho_x(F(\mathbf{X}_{t+s}) \ind{\zeta > t+s}|\zeta > t) = \int_{\mathfrak{P}} G({\bf x})\mathcal{D}^\rho(\dd {\bf x}) = \E_{\mathcal{D}^\rho}^\rho(F(\mathbf{X}_s )\ind{ \zeta > s}).
\]
On the other hand, using \eqref{hhsurvival} or Lemma~\ref{lem:yaglomLimit}, we have
\[
  \lim_{t \to \infty} \frac{\P^\rho_x(\zeta > t+s)}{\P_x^\rho(\zeta > t)} = e^{-(\rho^2/2 - 1)s}.
\]
Therefore, using that $\lim_{t \to \infty} \E^\rho_x(F(\mathbf{X}_{t+s})|\zeta > t+s) = \int F(\mathbf{x})\mathcal{D}^{\rho}(\dd \mathbf{x})$ by Proposition~\ref{prop:existenceYaglom} again, we deduce that
\[
\int_{\mathfrak{P}} F({\bf x})\mathcal{D}^\rho(\dd {\bf x}) = \E^\rho_{\mathcal{D}^\rho}(F(\mathbf{X}_s ) \ind{ \zeta > s}) e^{(\rho^2/2- 1)s},
\]
and hence that $\mathcal{D}^\rho$ is a quasi-stationary distribution of the \BBMabs{} with parameter $(\rho^2/2 - 1)$.

Using \eqref{theta}, we have
\begin{equation}
  \label{eqn:extinctionTime}
  \P^\rho_{\mathcal{D}^\rho}(\zeta > t) = e^{-(\frac{\rho^2}{2} - 1)t}.
\end{equation}
We now prove that the quasi-stationary distribution $\mathcal{D}^\rho$ is minimal. Let $\lambda > \frac{\rho^2}{2}-1$. Using \eqref{hhsurvival} and the branching property, we observe that for all $\nu \in \mathfrak{P}^*$, we have
\[
  \E_\nu^\rho(e^{\lambda \zeta}) = \infty.
\]
Let $\mu > \lambda$, and let us assume that there is a quasi-stationary measure $\mathcal{E}$ with parameter $\mu$.  We would have $\P_\mathcal{E}^\rho(\zeta > t) = e^{-\mu t}$, and therefore $\E_\mathcal{E}^\rho(e^{\lambda \zeta}) < \infty$. As a result, $\E_\nu^{\rho}(e^{\lambda \zeta}) < \infty$ for $\mathcal{E}$-a.s. $\nu \in \mathfrak{P}^*$, showing that $\mathcal{E}(\mathfrak{P}^*) = 0$, which leads to a contradiction.
\end{proof}

\section{Properties of the Yaglom limit}\label{Yaglomsec}

In this section, we prove Theorems \ref{NandM} and \ref{etachi}. We first recall in Section~\ref{subsec:ms22} the definitions and results from \cite{ms22} that will be used for these proofs. Then, using a coupling described in Section~\ref{couplesec} and preliminary bounds obtained in Section~\ref{subsec:bounds}, we prove Theorems~\ref{NandM} and~\ref{etachi} in Sections~\ref{subsec:NandM} and~\ref{subsec:etachi} respectively, by comparing \BBMabs{} with a branching Brownian motion in a strip of time-varying width described in Section~\ref{subsec:truncation}.

Throughout this section, we will be considering limits as $\eps \rightarrow 0$.  If $f$ and $g$ are functions, we write $f(\eps) \ll g(\eps)$ if $\lim_{\eps \rightarrow 0} f(\eps)/g(\eps) = 0$, $f(\eps) \sim g(\eps)$ if $\lim_{\eps \rightarrow 0} f(\eps)/g(\eps)=1$, and $f(\eps) \ll g(\eps)$ if there is a positive constant $C$ such that $f(\eps) \leq Cg(\eps)$ for all $\eps > 0$.

\subsection{Definitions of \texorpdfstring{$Z(\nu, L)$ and $T(\nu)$}{Z and T}}
\label{subsec:ms22}

In this subsection, we review some definitions and results from \cite{ms22} which are important for understanding the structure of the proofs.  The paper \cite{ms22} pertained to the case of critical drift $\rho = \sqrt{2}$ and used a different scaling in which the branching rate is $1/2$ rather than $1$ and the critical drift is $1$ rather than $\sqrt{2}$.  However, results from \cite{ms22} can easily be transferred to the setting of the present paper by scaling time by a factor of 2 and scaling space by a factor of $\sqrt{2}$.

For $t > 0$ and $s \in [0, t]$, let
\begin{equation}\label{Lcdef}
L_t(s) = c (t - s)^{1/3}, \qquad \text{where } c = \frac{(3 \pi^2)^{1/3}}{\sqrt{2}}.
\end{equation}
We also write $L_t = L_t(0)$.  For $L > 0$ and $x > 0$, define
\begin{equation}\label{zdef}
z(x,L) = \sqrt{2}L e^{\sqrt{2}(x - L)} \sin \bigg( \frac{\pi x}{L} \bigg) \1_{\{0 < x < L\}}.
\end{equation}
Then, for an arbitrary point measure $\nu \in \mathfrak{P}$ defined as in \eqref{nudef}, we define
$$Z(\nu, L) = \sum_i z(x_i, L), \qquad Z_t(s) = Z({\bf X}_s, L_t(s)).$$
We note that $Z_t(s)$ can be viewed as a measure of the ``size" of the process at time $s$ and, in particular, can be used to predict how likely it is that the process will survive until time $t$.  For example, it is shown in \cite{bbs14} that there are positive constants $C_3$ and $C_4$ such that for all $x$ such that $0 < x \leq L_t - 1$, we have
\begin{equation}\label{bbs14survival}
C_3 z(x,L_t) \leq \P_x^{\sqrt{2}}(\zeta > t) \leq C_4 z(x,L_t).
\end{equation}

This result allows us to define a quantity, which we call $T(\nu)$, which predicts with high accuracy how long the BBM with absorption will survive, if it begins from the initial configuration $\nu \in \mathfrak{P}$.  Following \cite{ms22}, we define
\begin{equation}\label{Tdef}
T(\nu) = \inf\{t: L_t \geq M(\nu) + 2 \mbox{ and }Z(\nu, L_t) \leq 1/2\}.
\end{equation}
That is, $T(\nu)$ is the smallest value of $t$ for which $Z(\nu, L_t) \leq 1/2$, considering only values of $t$ that are large enough that all particles in the initial configuration $\nu$ are below $L_t - 2$.  As explained in \cite{ms22}, the quantity $T(\nu)$ is well-defined and finite.  Note that $t\mapsto L_t$ is increasing and continuous, and $L \mapsto Z(\nu, L)$ is continuous for $L \geq M(\nu) + 2$.  Therefore, when $t = T(\nu)$, one of the two inequalities in the definition of $T(\nu)$ must be an equality.  That is, either $L_{T(\nu)} = M(\nu) + 2$, which will be the case if $Z(\nu, M(\nu) + 2) < 1/2$, or else $Z(\nu, L_{T(\nu)}) = 1/2$.

The result \eqref{bbs14survival} suggests that when $Z(\nu, L_t) = 1/2$, the probability that the process survives until time $t$ will not be close to either zero or one, and therefore the process should survive until approximately time $t$.  The following more formal statement, which is Lemma 2.11 of \cite{ms22}, states that this intuition is correct as long as no particle of $\nu$ is very close to $L_{T(\nu)}$.

\begin{lemma}\label{lemma2.11}
Let $\delta > 0$.  There exist positive constants $k$, $t_0$, and $a$ such that for all $\nu \in \mathfrak{P}$ such that $T(\nu) \geq t_0$ and $L_{T(\nu)} - M(\nu) \geq a$, we have $$\P^{\sqrt{2}}_{\nu} \big( |\zeta - T(\nu)| \leq k T(\nu)^{2/3} \big) > 1 - \delta.$$
\end{lemma}

A challenge when applying Lemma \ref{lemma2.11} is proving that $L_{T(\nu)} - M(\nu) \geq a$ with high probability.  However, the following result, which is Lemma 2.12 of \cite{ms22}, shows that if we begin with any initial configuration, the configuration of particles after a short time will satisfy this condition with high probability.

\begin{lemma}\label{lemma2.12}
Let $\delta > 0$ and $A > 0$.  There exist positive real numbers $t_1$ and $d$, depending on $\delta$ and $A$, such that if $\nu \in \mathfrak{P}$, then
$$\P_{\nu}^{\sqrt{2}}(\{M({\bf X}_d) \geq L_{T({\bf X}_d)} - A\} \cap \{T({\bf X}_d) \geq t_1\}) < \delta.$$
\end{lemma}

\subsection{A coupling}\label{couplesec}

To apply Lemmas \ref{lemma2.11} and \ref{lemma2.12}, it will be necessary to show that the branching Brownian motion with drift $-\sqrt{2}$ is a good approximation to the branching Brownian motion with drift $-\rho = -\sqrt{2} - \eps$, as long as $\eps$ is small.  To do this, we construct a coupling so that, up to some time $t$, the \BBMabs{} stays in between two copies of BBM${}^+(-\sqrt{2})$ started from different initial configurations.

\begin{definition}\label{dominates}
Let $\nu_1$ and $\nu_2$ be point measures in $\mathfrak{P}$.  We say that $\nu_1$ {\em dominates} $\nu_2$, and write $\nu_1 \succ \nu_2$, if for all $y > 0$, we have $\nu_1([y, \infty)) \geq \nu_2([y, \infty))$.
\end{definition}

\begin{definition}\label{numinusb}
Let $\nu = \sum_i \delta_{x_i} \in \mathfrak{P}$, and let $\Delta > 0$.  Let $\nu - \Delta = \sum_i \delta_{x_i - \Delta}\1_{\{x_i > \Delta\}}$ be the configuration of particles obtained from $\nu$ by removing all particles in $(0, \Delta]$ and moving all others to the left by $\Delta$.
\end{definition}

To construct the coupling, we begin with a process $({\bf X}^*_t, t \geq 0)$ which evolves, under ${\bf P}_{\nu}^{\rho}$, as a BBM without drift and without absorption at zero, started from ${\bf X}^*_0 = \nu$.  Let $\mathcal{N}_t^*$ denote the set of particles alive at time $t$.  For all $u \in \mathcal{N}_t^*$ and $s \leq t$, we write $X^*_s(u)$ for the position at time $s$ of the particle $u$.  For all $t \geq 0$ and $\Delta > 0$, we define
\begin{align*}
{\bf X}_t &= \sum_{u \in \mathcal{N}_t^*} \delta_{X_t^*(u) - \rho t} \1_{\{X_s^*(u) - \rho s \geq 0 \: \forall s \in [0, t]\}} \\
{\bf X}^U_t &= \sum_{u \in \mathcal{N}_t^*} \delta_{X_t^*(u) - \sqrt{2} t} \1_{\{X_s^*(u) - \sqrt{2} s \geq 0 \: \forall s \in [0, t]\}} \\
{\bf X}^{L,\Delta}_t &= \sum_{u \in \mathcal{N}_t^*} \delta_{X_t^*(u) - \Delta - \sqrt{2} t} \1_{\{X_s^*(u) - \Delta - \sqrt{2} s \geq 0 \: \forall s \in [0, t]\}}.
\end{align*}
Under ${\bf P}_{\nu}^{\rho}$, the process $({\bf X}_t, t \geq 0)$ is \BBMabs, started from ${\bf X}_0 = \nu$.  The process $({\bf X}^U_t, t \geq 0)$ is BBM${}^{+}(-\sqrt{2})$, started from ${\bf X}^U_0 = \nu$.  The process $({\bf X}^{L,\Delta}_t, t \geq 0)$ is BBM${}^{+}(-\sqrt{2})$, started from ${\bf X}^{L,\Delta}_0 = \nu - \Delta$.  It is not difficult to see that if ${\bf X}_t$ has a particle at $y$, then ${\bf X}_t^U$ has a particle at $y + \eps t$ because the particles have a drift of $-\sqrt{2}$ instead of $-\rho$.  Also, ${\bf X}_t^U$ may have additional particles for which there is no corresponding particle alive at time $t$ in ${\bf X}_t$ because the smaller drift allowed the particle to avoid being absorbed at the origin.  Therefore,
\begin{equation}\label{Ucouple}
{\bf X}_t \prec {\bf X}_t^U \quad \mbox{for all }t \geq 0.
\end{equation}
Likewise, if $t \leq \eps^{-1} \Delta$ and ${\bf X}_t^{L,\Delta}$ has a particle at $y$, then ${\bf X}_t$ has a particle at $y + \Delta - \eps t$.  Thus,
\begin{equation}\label{Lcouple}
{\bf X}_t \succ {\bf X}_t^{L,\Delta} \quad \mbox{for all }t \in [0, \eps^{-1} \Delta].
\end{equation}

Define the extinction times $$\zeta^U = \inf\{t \geq 0: {\bf X}_t^U(\R^+) = 0\}, \qquad \zeta^{L,\Delta} = \inf\{t \geq 0: {\bf X}_t^{L,\Delta}(\R^+) = 0\}.$$ Note that \eqref{Ucouple} and \eqref{Lcouple} imply that
\begin{equation}\label{Uextinct}
\min(\epsilon^{-1}\Delta, \zeta^{L,\Delta}) \leq \zeta \leq \zeta^U.
\end{equation}

\subsection{Preliminary bounds on \texorpdfstring{$T(D^{\rho})$ and $M(D^{\rho})$}{T and M}}
\label{subsec:bounds}

Our goal in this subsection is to prove lemmas \ref{T0upper} and \ref{M0bound}, which give preliminary bounds for the quantities $T(D^{\rho})$ and $M(D^{\rho})$.  Lemma \ref{T0upper} shows that $T(D^{\rho})$ is at least comparable to $\eps^{-1}$, while Lemma \ref{M0bound} shows that $M(D^{\rho})$ is at most of the order $\eps^{-1/3}$.  These results will be used to obtain more precise estimates later.  The proofs of both of these results involve making comparisons to BBM with critical drift.

Lemma \ref{derzlemma} will help us to determine how much $Z_t({\bf X}_s)$ changes when the location of the particles or the upper boundary moves slightly.  The key ideas in the proof come from the proof of Lemma~4.1 in \cite{ms22}.  We will also use the inequalities
\begin{equation}\label{sinbounds}
\begin{array}{ll} (2/\pi)y \leq \sin(y) \leq y & \mbox{ if } 0 < y \leq \pi/2, \\
(2/\pi)(\pi - y) \leq \sin(y) \leq \pi - y & \mbox{ if } \pi/2 \leq y < \pi.
\end{array}
\end{equation}

\begin{lemma}\label{derzlemma}
If $1 \leq x \leq L - 2$, then
\begin{equation}\label{dzx}
z(x,L) \Big( \sqrt{2} - \frac{\pi}{4} \Big) \leq \frac{\partial z}{\partial x} (x,L) \leq z(x,L) \Big( \sqrt{2} + \frac{\pi}{2} \Big),
\end{equation}
and the lower bound in (\ref{dzx}) holds also when $0 < x < 1$.
Also, if $0 < x \leq L - 2$, then
\begin{equation}\label{dLx}
\frac{\partial z}{\partial L} (x,L) \leq -\frac{1}{8} z(x, L).
\end{equation}
\end{lemma}

\begin{proof}
A short calculation gives
$$\frac{\partial z}{\partial x} (x,L) = z(x,L) \bigg( \sqrt{2} + \frac{\pi}{L} \cdot \frac{\cos(\pi x/L)}{\sin(\pi x/L)} \bigg).$$
The upper bound in (\ref{dzx}) holds because $\cos(\pi x/L) \leq 1$ and, when $1 \leq x \leq L - 2$, we have $\sin(\pi x/L) \geq 2/L$.
When $0 < x \leq L/2$, the lower bound in (\ref{dzx}) follows immediately because $\cos(\pi x/L)/\sin(\pi x/L) \geq 0$.  When $L/2 < x \leq L-2$, we have $|\cos(\pi x/L)| \leq 1$ and $\sin(\pi x/L) \geq 4/L$, which together give the lower bound in (\ref{dzx}).

Likewise, we calculate
$$\frac{\partial z}{\partial L} (x,L) = -z(x,L) \bigg(\sqrt{2} + \frac{\pi x}{L^2} \cdot \frac{\cos(\pi x/L)}{\sin(\pi x/L)} - \frac{1}{L} \bigg).$$
Because $x \leq L$, the argument used to prove (\ref{dzx}) yields that for $0 < x \leq L-2$, we have
$$\frac{\partial z}{\partial L} (x,L) \leq -z(x,L) \bigg(\sqrt{2} - \frac{\pi}{4} - \frac{1}{L} \bigg).$$
If $0 < x \leq L - 2$, then we must have $L \geq 2$ and therefore $\sqrt{2} - \pi/4 - 1/L \geq \sqrt{2} - \pi/4 - 1/2 > 1/8$, which proves (\ref{dLx}).
\end{proof}

\begin{lemma}\label{critsurvive2}
For all $t > 0$, all $\rho \geq \sqrt{2}$, and all initial configurations $\nu$ such that $M(\nu) \leq L_t - 1$, we have
$$\P_{\nu}^{\rho}(\zeta > t) \leq C_4 Z_t(0).$$
\end{lemma}

\begin{proof}
By the coupling defined in Section \ref{couplesec}, we observe that $\rho \mapsto \P^\rho_\nu(\zeta > t)$ is non-increasing, as survival gets more difficult as the drift increases. It is therefore enough to prove the inequality with $\rho = \sqrt{2}$, which follows immediately from \eqref{bbs14survival} and a union bound.
\end{proof}

\begin{lemma}\label{T0upper}
For all $\eta > 0$, there exists $k > 0$ such that if $\eps$ is sufficiently small, then
$$\P(T(D^{\rho}) < k \eps^{-1}) < \eta.$$
\end{lemma}

\begin{proof}
By the definition of $T(D^{\rho})$, we have $Z(D^{\rho}, L_{T(D^{\rho})}) \leq 1/2$.  Note that on the event $\{T(D^{\rho}) < k \eps^{-1}\}$, we have $L_{2 k \eps^{-1}} - L_{T(D^{\rho})} \geq c (2^{1/3} - 1) k^{1/3} \eps^{-1/3}$.  It follows from (\ref{dLx}) that on $\{T(D^{\rho}) < k \eps^{-1}\}$, we have $$Z(D^{\rho}, L_{2 k \eps^{-1}}) \leq Z(D^{\rho}, L_{T(D^{\rho})}) \cdot e^{-(1/8)(L_{2 k \eps^{-1}} - L_{T(D^{\rho})})} \leq \frac{1}{2} e^{-(1/8) c (2^{1/3} - 1) k^{1/3} \eps^{-1/3}}.$$
Therefore, by Lemma \ref{critsurvive2}, since ${\bf X}_0$ has the same law as $D^{\rho}$ under $\P^{\rho}_{\mathcal{D}^{\rho}}$, we have
$$\P_{\mathcal{D}^{\rho}}^{\rho}(\zeta > 2 k \eps^{-1} \,|\, T({\bf X}_0) < k \eps^{-1}) \leq \frac{C_4}{2} e^{-(1/8) c (2^{1/3} - 1) k^{1/3} \eps^{-1/3}}.$$  It follows that
$$\P_{\mathcal{D}^{\rho}}^{\rho}(\zeta \leq 2 k \eps^{-1}) \geq \P(T(D^{\rho}) < k \eps^{-1}) \bigg(1 - \frac{C_4}{2} e^{-(1/8) c (2^{1/3} - 1) k^{1/3} \eps^{-1/3}} \bigg).$$
By Corollary \ref{survivalcor}, $$\P_{\mathcal{D}^{\rho}}^{\rho}(\zeta \leq 2 k \eps^{-1}) = 1 - \exp((1 - \rho^2/2)(2 k \eps^{-1})) = 1 - \exp(-(2 \sqrt{2} k + k \eps)) \leq 2 \sqrt{2} k + k \eps,$$ so
$$\P(T(D^{\rho}) < k \eps^{-1}) \leq (2 \sqrt{2} k + k \eps) \bigg(1 - \frac{C_4}{2} e^{-(1/8) c (2^{1/3} - 1) k^{1/3} \eps^{-1/3}} \bigg)^{-1}.$$
Therefore, if we take $k < \eta/(2 \sqrt{2})$, then
$$\limsup_{\eps \rightarrow 0} \P(T(D^{\rho}) < k \eps^{-1}) \leq 2 \sqrt{2} k < \eta,$$ which implies the lemma.
\end{proof}

For the following lemma, we will need a result for the case of critical drift.  Maillard and Schweinsberg \cite{ms22}, building on earlier work in \cite{kesten, bbs14}, showed that for all $x \in \R$,
\begin{equation}\label{critsurvive}
\lim_{t \rightarrow \infty} \P_{L_t + x}^{\sqrt{2}}(\zeta > t) = \phi(x),
\end{equation}
where $\phi$ is a function satisfying $\lim_{x \rightarrow \infty} \phi(x) = 1$ and $\lim_{x \rightarrow -\infty} \phi(x) = 0$.

\begin{lemma}\label{M0bound}
For all $\eta > 0$, there exists $J > 0$ such that if $\eps$ is sufficiently small, then $$\P(M(D^{\rho}) > J \eps^{-1/3}) < \eta.$$
\end{lemma}

\begin{proof}
By \eqref{Uextinct} with $\Delta = \eps t$, we have $$\P_x^{\rho}(\zeta > t) \geq \P_{x - \eps t}^{\sqrt{2}}(\zeta > t).$$
Setting $x = J \eps^{-1/3}$ and $t = J \eps^{-1}$, we get $$\P_{J \eps^{-1/3}}^{\rho}(\zeta > J \eps^{-1}) \geq \P_{J \eps^{-1/3} - J}^{\sqrt{2}}(\zeta > J \eps^{-1}).$$
Choose $a$ such that $0 < a < \phi(0)$, where $\phi$ is the function from (\ref{critsurvive}).  By (\ref{critsurvive}), there is an $\eps_0$ such that if $0 < \eps < \eps_0$, then
$$\P^{\sqrt{2}}_{cJ^{1/3} \eps^{-1/3}}(\zeta > J \eps^{-1}) > a.$$
Therefore, if we choose $J$ large enough that $J \eps^{-1/3} - J > cJ^{1/3} \eps^{-1/3}$ for $0 < \eps < \eps_0$, then
$$\P_{J \eps^{-1/3}}^{\rho}(\zeta > J \eps^{-1}) > a$$
for $0 < \eps < \eps_0$.
It follows that for $0 < \eps < \eps_0$, $$\P_{\mathcal{D}^{\rho}}^{\rho}(\zeta > J \eps^{-1}) \geq a \P(M(D^{\rho}) > J \eps^{-1/3}).$$
Corollary \ref{survivalcor} implies that $$\lim_{J \rightarrow \infty} \lim_{\eps \rightarrow 0} \P_{\mathcal{D}^{\rho}}^{\rho}(\zeta > J \eps^{-1}) = 0,$$
so we also must have $$\lim_{J \rightarrow \infty} \limsup_{\eps \rightarrow 0} \P(M(D^{\rho}) > J \eps^{-1/3}) = 0,$$
which implies the statement of the lemma.
\end{proof}

\subsection{Comparison between \texorpdfstring{$M(D^{\rho})$ and $L_{T(D^{\rho})}$}{M and L}}

Our goal in this subsection is to prove Lemma \ref{biglemma}, which establishes that under the quasi-stationary distribution we constructed, the right-most particle $M(D^{\rho})$ is not too close to $L_{T(D^{\rho})}$.  This result is necessary for applying Lemma \ref{lemma2.11}.

\begin{lemma}\label{Tmonotone}
If $\nu_1 \prec \nu_2$, then $T(\nu_1) \leq T(\nu_2)$.
\end{lemma}

\begin{proof}
It is immediate from the definitions that $M(\nu_1) \leq M(\nu_2)$.  Therefore, if $L_{T(\nu_1)} \leq M(\nu_2) + 2$, then $T(\nu_1) \leq T(\nu_2)$.  On the other hand, for all $t$ such that $L_t > M(\nu_2) + 2$, the fact that the lower bound in \eqref{dzx} holds for all $x \in (0, L - 2)$ implies that $Z(\nu_1, L_t) \leq Z(\nu_2, L_t)$.  Therefore, the conclusion that $T(\nu_1) \leq T(\nu_2)$ also must hold when $L_{T(\nu_1)} > M(\nu_2) + 2$.
\end{proof}

\begin{lemma}\label{biglemma}
For all $A > 0$, we have
\begin{equation}\label{RDrho}
\lim_{\eps \rightarrow 0} \P(M(D^{\rho}) \geq L_{T(D^{\rho})} - A) = 0.
\end{equation}
\end{lemma}

\begin{proof}
Let $\delta > 0$.  Now choose $t_1$ and $d$ as in Lemma \ref{lemma2.12}.  Recall from the coupling in Section~\ref{couplesec} that under $\P_{\mathcal{D}^{\rho}}^{\rho}$, the process $({\bf X}_t^{L, \eps d}, t \geq 0)$ is a BBM${}^+(-\sqrt{2})$ whose initial distribution has the same law as $D^{\rho} - \eps d$.  By Lemma \ref{lemma2.12},
\begin{equation}\label{fromms}
\P^{\rho}_{\mathcal{D}^{\rho}}(\{M({\bf X}_d^{L, \eps d}) \geq L_{T({\bf X}_d^{L, \eps d})} - A\} \cap \{T({\bf X}_d^{L, \eps d}) \geq t_1\}) < \delta.
\end{equation}
Now \eqref{Lcouple} implies that ${\bf X}_d \succ {\bf X}_d^{L, \eps d}$.  Therefore, by Lemma \ref{Tmonotone}, we have $T({\bf X}_d^{L, \eps d}) \leq T({\bf X}_d)$ and thus $L_{T({\bf X}_d^{L, \eps d})} \leq L_{T({\bf X}_d)}$.  Now
\begin{align}\label{3terms}
&\P^{\rho}_{\mathcal{D}^{\rho}}(M({\bf X}_d) \geq L_{T({\bf X}_d)} - A) \nonumber \\
&\quad \leq \P^{\rho}_{\mathcal{D}^{\rho}}(T({\bf X}_d) \leq \eps^{-7/8}) \nonumber \\
&\quad \quad + \P^{\rho}_{\mathcal{D}^{\rho}}(\{T({\bf X}_d) > \eps^{-7/8}\} \cap \{M({\bf X}_d) \neq M({\bf X}_d^{L, \eps d})\} \cap \{M({\bf X}_d) \geq L_{T({\bf X}_d)} - A\}) \nonumber \\
&\quad \quad \quad + \P^{\rho}_{{\cal D}^\rho}(\{T({\bf X}_d) > \eps^{-7/8}\} \cap \{M({\bf X}_d) = M({\bf X}_d^{L, \eps d})\} \cap \{M({\bf X}_d) \geq L_{T({\bf X}_d)} - A\}).
\end{align}

We separately bound these three terms.  By Corollary \ref{survivalcor}, under $\P^{\rho}_{\mathcal{D}^{\rho}}$, the total variation distance between $T({\bf X}_0)$ and $T({\bf X}_d)$ is at most $1 - e^{-(\rho^2/2 - 1)d}$.  Therefore, by Lemma \ref{T0upper},
\begin{align}\label{term1}
\limsup_{\eps \rightarrow 0} \P^{\rho}_{\mathcal{D}^{\rho}}(T({\bf X}_d) \leq \eps^{-7/8}) &\leq \limsup_{\eps \rightarrow 0} \big( \P^{\rho}_{\mathcal{D}^{\rho}}(T({\bf X}_0) \leq \eps^{-7/8}) + (1 - e^{-(\rho^2/2 - 1)d}) \big) \nonumber \\
&= \limsup_{\eps \rightarrow 0} \P(T(D_{\rho}) \leq \eps^{-7/8}) \nonumber \\
&= 0.
\end{align}

To bound the second term, note that the particles in ${\bf X}_d^{L, \eps d}$ are the same as the particles in ${\bf X}_d$, except that some particles may be missing from the former configuration because of additional killing at the origin.  Any particle that is in ${\bf X}_d$ but not in ${\bf X}_d^{L, \eps d}$ must have an ancestor that passes within $\eps d$ of the origin.  Therefore, if $M({\bf X}_d) \neq M({\bf X}_d^{L, \eps d})$, then the particle located at $M({\bf X}_d)$ must have an ancestor that passes within $\eps d$ of the origin between time $0$ and time $d$.  In particular, if $M({\bf X}_d) \geq L_{T({\bf X}_d)} - A$ and $T({\bf X}_d) > \eps^{-7/8}$, then there must be a particle which gets to a position $\eps d$ or smaller, and then increases by at least $c \eps^{-7/24} - A - \eps d$, which for $c_0 < c$ and sufficiently small $\eps$ is at least $c_0 \eps^{-7/24}$.  If a particle is below $\eps d$, then standard Gaussian tail estimates imply that there are positive constants $C$ and $C'$ such that the probability that this particle increases by at least $c_0 \eps^{-7/24}$ by time $d$ is at most $C e^{-C' \eps^{-7/12}}$.  By the many-to-one lemma, accounting for branching multiplies this probability by at most the constant factor of $e^d$.  By \eqref{meanN}, the expected number of particles at time zero is bounded above by $e^{C_2 \eps^{-1/2}}$.  Using Markov's inequality, it follows that
\begin{align}\label{term2}
&\limsup_{\eps \rightarrow 0} \P^{\rho}_{\mathcal{D}^{\rho}}(\{T({\bf X}_d) > \eps^{-7/8}\} \cap \{M({\bf X}_d) \neq M({\bf X}_d^{L, \eps d})\} \cap \{M({\bf X}_d) \geq L_{T({\bf X}_d)} - A\}) \nonumber \\
&\hspace{3.4in}\leq \limsup_{\eps \rightarrow 0} Ce^d e^{-C' \eps^{-7/12}} e^{C_2 \eps^{-1/2}} = 0.
\end{align}

It remains to bound the third term.  Because $M({\bf X}_d^{L, \eps d}) \leq T({\bf X}_d^{L, \eps d}) + 2$ by definition, on the event $\{T({\bf X}_d) > \eps^{-7/8}\} \cap \{M({\bf X}_d) = M({\bf X}_d^{L, \eps d})\} \cap \{M({\bf X}_d) \geq L_{T({\bf X}_d)} - A\}$, we have for sufficiently small $\eps$, $$T({\bf X}_d^{L, \eps d}) \geq M({\bf X}_d^{L, \eps d}) - 2 = M({\bf X}_d) - 2 \geq L_{T({\bf X}_d)} - A - 2 \geq c \eps^{-7/24} - A - 2 \geq t_1.$$  Recalling also that $L_{T({\bf X}_d^{L, \eps d})} \leq L_{T({\bf X}_d)}$, it now follows from Lemma \ref{lemma2.12} that
\begin{align}\label{term3}
&\limsup_{\eps \rightarrow 0} \P^{\rho}_{{\cal D}^\rho}(\{T({\bf X}_d) > \eps^{-7/8}\} \cap \{M({\bf X}_d) = M({\bf X}_d^{L, \eps d})\} \cap \{M({\bf X}_d) \geq L_{T({\bf X}_d)} - A\}) \nonumber \\
&\hspace{1in}\leq \limsup_{\eps \rightarrow 0} \P^{\rho}_{{\cal D}^\rho}( \{M({\bf X}_d^{L, \eps d}) \geq  L_{T({\bf X}_d^{L, \eps d})} - A\} \cap \{T({\bf X}_d^{L, \eps d}) \geq t_1\}) \leq \delta.
\end{align}

It now follows from \eqref{3terms}, \eqref{term1}, \eqref{term2}, and \eqref{term3} that
$$\limsup_{\eps \rightarrow 0} \P^{\rho}_{\mathcal{D}^{\rho}}(M({\bf X}_d) \geq L_{T({\bf X}_d)} - A) \leq \delta.$$
By Corollary \ref{survivalcor},
$$\P(M(D^{\rho}) \geq L_{T(D^{\rho})} - A) \leq \P^{\rho}_{\mathcal{D}^{\rho}}(M({\bf X}_d) \geq L_{T({\bf X}_d)} - A) + (1 - e^{-(\rho^2/2 - 1)d}).$$
It follows that $$\limsup_{\eps \rightarrow 0} \P(M(D^{\rho}) \geq L_{T(D^{\rho})} - A) \leq \delta.$$
Because $\delta > 0$ was arbitrary, the result follows.
\end{proof}

\subsection{A truncated process}
\label{subsec:truncation}

Given $t > 0$, it will often be useful to consider a truncated process in which particles at time $s \in [0, t]$ are killed not only if they reach $0$ but also if they reach an upper boundary at $L_t(s)$.  Note that the process is guaranteed to die out before time $t$ because $L_t(t) = 0$.   We will denote probabilities and expectations for this process when the initial condition has law $\mathcal{D}$ and the drift is $-\rho$ by $\P^{\rho, t}_{\mathcal{D}}$ and $\P^{\rho, t}_{\mathcal{D}}$.  As in the case when particles are absorbed only at the origin, for $\nu \in \mathfrak{P}$ and $x > 0$, we write the subscripts $\nu$ and $x$ in place of $\delta_{\nu}$ and $\delta_{\delta_x}$.  Also, when $0 < r < t$, we will write $\P^{\rho, t}_{x,r}$ and $\E^{\rho, t}_{x,r}$ when the process starts from one particle at time $r$ located at $x$.

For $0 \leq r < s < t$, we denote by $q_{r,s}^{\rho, t}(x, y)$ the density for the process started from one particle at $x$ at time $r$, meaning that if $A$ is a Borel subset of $(0, L_t(s))$, then the expected number of particles in the set $A$ at time $t$ is $\int_A q_{r,s}^{\rho, t}(x,y) \: dy$.  We will write $q_s^{\rho, t}(x,y)$ in place of $q_{0,s}^{\rho, t}(x,y)$.  We review here some results from \cite{ms22} in the critical case $\rho = \sqrt{2}$.  Let $w_s(x,y)$ be the density for a single Brownian particle killed at $0$ and $1$, meaning that if a Brownian particle starts at $x \in (0, 1)$ and is killed upon reaching $0$ or $1$, then the probability that it is in some Borel subset $A$ of $(0,1)$ at time $s$ is given by $\int_A w_s(x,y) \: dy$.  According to Lemma 5 in \cite{bbs13} or equations (5.3) and (5.5) of \cite{ms22},
\begin{equation}\label{wseq}
w_s(x,y) = 2 e^{-\pi^2 s/2} \sin(\pi x) \sin(\pi y)(1 + d_s(x,y)),
\end{equation}
where
\begin{equation}\label{dbound}
|d_s(x,y)| \leq \sum_{n=2}^{\infty} n^2 e^{-\pi^2 (n^2 - 1)s/2}.
\end{equation}
For $0 \leq r < s < t$, define
\begin{equation}\label{taudef}
\tau_t(r,s) = \int_r^s \frac{1}{L_t(u)^2} \: du = \frac{2}{\pi^2} \cdot \sqrt{2}(L_t(r) - L_t(s)).
\end{equation}
Then, Proposition 5.4 of \cite{ms22} in the case $A = 0$, along with the Girsanov transformation explained at the beginning of section 5.2 of \cite{ms22}, give that for $0 \leq r < s < t$, $0 < x < L_t(r)$, and $0 < y < L_t(s)$, we have
\begin{equation}\label{qrs}
q_{r,s}^{\sqrt{2}, t}(x,y) = \frac{e^{O((t - s)^{-1/3})}}{L_t(r)^{1/2} L_t(s)^{1/2}} e^{\sqrt{2}(x - y)} w_{\tau_t(r,s)}\bigg( \frac{x}{L_t(r)}, \frac{y}{L_t(s)} \bigg).
\end{equation}
A consequence of Girsanov's Theorem is that if $\rho = \sqrt{2} + \eps$, then
\begin{equation}\label{qGirsanov}
q_{r,s}^{\rho, t}(x,y) = q_{r,s}^{\sqrt{2}, t}(x,y) \cdot e^{\eps(x - y) -\sqrt{2} \eps (s - r) - \eps^2 (s - r)/2}.
\end{equation}
The following result follows immediately from (\ref{taudef}), (\ref{qrs}) and (\ref{qGirsanov}).  Here we assume $\rho = \sqrt{2} + \eps$, where $\eps > 0$.

\begin{lemma}\label{densitybound}
Fix $\delta > 0$, and $\beta > 0$.  There exist positive numbers $t_0$ and $C_5$ such that for all $t \geq t_0$, all $r$ and $s$ such that $s \leq (1 - \delta) t$, $\delta t^{2/3} \leq s - r \leq \delta^{-1} t^{2/3}$, and $\eps(s - r) \leq \beta$, we have
$$q_{r,s}^{\rho, t}(x,y) \leq \frac{C_5}{L_t} e^{\sqrt{2} x} \sin \bigg( \frac{\pi x}{L_t(r)} \bigg) e^{-\sqrt{2} y} \sin \bigg( \frac{\pi y}{L_t(s)} \bigg)$$
for all $x$ and $y$ such that $0 < x < L_t(r)$ and $0 < y < L_t(s)$.
If instead we let $t \rightarrow \infty$ and allow $\eps$, $r$, and $s$ to depend on $t$, then if $s \ll t$, $s - r \gg t^{2/3}$, and $\eps (s - r) \rightarrow 0$, we have $$q_{r,s}^{\rho, t}(x,y) = \frac{2(1 + o(1))}{L_t} e^{\sqrt{2}(L_t(s) - L_t(r))} e^{\sqrt{2} x} \sin \bigg( \frac{\pi x}{L_t(r)} \bigg) e^{-\sqrt{2} y} \sin \bigg( \frac{\pi y}{L_t(s)} \bigg)$$
for all $x$ and $y$ such that $0 < x < L_t(r)$ and $0 < y < L_t(s)$.
\end{lemma}

Let $$y(x,L) = \frac{x}{L} e^{\sqrt{2}(x - L)}, \qquad Y_t(s) = \sum_{u \in {\cal N}_s} y(X_s(u), L_t(s)).$$
We denote by $R_t(s)$ the number of particles killed at the upper boundary between times $0$ and $s$.

\begin{lemma}\label{meanR}
Fix $\delta > 0$ and $\beta > 0$.  Then there exist positive real numbers $t_0$ and $C_6$ such that
for all $t \geq t_0$, all $s$ such that $\delta t^{2/3} \leq s \leq (1 - \delta) t$, all $\rho$ such that $0 \leq (\rho - \sqrt{2})s \leq \beta$, and all initial configurations $\nu \in \mathfrak{P}$ for which $\nu([L_t, \infty)) = 0$, we have
$$\E_{\nu}^{\rho, t}[R_t(s)] \leq C_6 \bigg( \frac{s}{t} Z_t(0) + Y_t(0) \bigg).$$
\end{lemma}

\begin{proof}
When $\rho =\sqrt{2}$, the result follows from the first inequality in Lemma~5.8 of \cite{ms22} and equation (\ref{taudef}).
Now suppose instead $\rho = \sqrt{2} + \eps$, where $0 < \eps s \leq \beta$.  It follows from (\ref{qGirsanov}) and the assumptions on $s$ that there are positive constants $C_7$ and $C_8$ such that for all $x \in (0, L_t)$ and all $y \in (0, L_t(s))$, we have
\begin{equation}\label{qsrhosqrt(2)}
C_7 q_s^{\sqrt{2}, t}(x,y) \leq q_s^{\rho, t}(x,y) \leq C_8 q_s^{\sqrt{2}, t}(x,y).
\end{equation}
Because $\E_{\nu}^{\rho, t}[R_t(s)]$ is the expected number of particle trajectories that hit the upper boundary before time $s$, the result follows from (\ref{qsrhosqrt(2)}) and the result when $\rho = \sqrt{2}$.
\end{proof}

\begin{lemma}\label{momlem}
Let $f: [0, \infty) \rightarrow [0, \infty)$ and $g: [0,1] \rightarrow [0, \infty)$ be bounded nonzero measurable functions.  Let $s > 0$, $t > 0$, and $\nu \in \mathfrak{P}$ depend on $\eps$ in such a way that as $\eps \rightarrow 0$, we have $t \rightarrow \infty$, $t^{2/3} \ll s \ll t$, $\eps s \rightarrow 0$, and $\nu([L_t, \infty)) = 0$.  Then the following hold as $\eps \rightarrow 0$:
\begin{align}
\E_{\nu}^{\rho, t} \bigg[\sum_{u \in {\cal N}^+_s} f(X_s(u)) \bigg] &\sim \frac{\pi}{\sqrt{2} L_t^3} e^{\sqrt{2} L_t(s)} Z_t(0) \int_0^{\infty} 2y e^{-\sqrt{2} y} f(y) \: dy \label{mom1} \\
\E_{\nu}^{\rho, t} \bigg[ \sum_{u \in {\cal N}^+_s} e^{\sqrt{2} X_s(u)} g \bigg( \frac{X_s(u)}{L_t(s)} \bigg) \bigg] &\sim \frac{2\sqrt{2}}{\pi L_t} e^{\sqrt{2} L_t(s)} Z_t(0) \int_0^1 \frac{\pi}{2} \sin(\pi y) g(y) \: dy \label{mom2} \\
\Var_{\nu}^{\rho, t} \bigg(\sum_{u \in {\cal N}^+_s} f(X_s(u)) \bigg) &\lesssim \frac{e^{2 \sqrt{2} L_t(s)}}{L_t^6} \bigg( \frac{s}{t} Z_t(0) + \frac{1}{L_t} Z_t(0) + Y_t(0) \bigg) \label{mom3} \\
\Var_{\nu}^{\rho, t} \bigg( \sum_{u \in {\cal N}^+_s} e^{\sqrt{2} X_s(u)} g \bigg( \frac{X_s(u)}{L_t(s)} \bigg) \bigg) &\lesssim \frac{e^{2 \sqrt{2} L_t(s)}}{L_t^2} \bigg( \frac{s}{t} Z_t(0) + \frac{\log L_t}{L_t} Z_t(0) + Y_t(0) \bigg). \label{mom4}
\end{align}
\end{lemma}

The proof of Lemma \ref{momlem} will be deferred until section \ref{momproof}.  It will be clear from the proof that the results (\ref{mom3}) and (\ref{mom4}) hold under the weaker hypotheses that $t^{2/3} \lesssim s$, $t - s \gtrsim t$, and $\eps s \lesssim 1$, although we will not need this stronger result in this paper.

\subsection{Proof of Theorem \ref{etachi}}
\label{subsec:etachi}

In this subsection, we complete the proof of Theorem \ref{etachi}.  We begin with the following consequence of Lemma \ref{momlem}.

\begin{lemma}
Let $f: [0, \infty) \rightarrow [0, \infty)$ and $g: [0,1] \rightarrow [0, \infty)$ be bounded nonzero measurable functions.  Let $s = \eps^{5/6}$.  Then, under $\P_{\mathcal{D}^{\rho}}^{\rho}$, we have as $\eps \rightarrow 0$,
\begin{equation}\label{probf}
L_{T({\bf X}_0)}^3 e^{-\sqrt{2} L_{T({\bf X}_0)}(s)} \sum_{u \in {\cal N}^+_{s}} f(X_s(u)) \rightarrow_p \frac{\pi}{2 \sqrt{2}} \int_0^{\infty}  2ye^{-\sqrt{2} y} f(y) \: dy
\end{equation}
and
\begin{equation}\label{probg}
L_{T({\bf X}_0)} e^{-\sqrt{2} L_{T({\bf X}_0)}(s)} \sum_{u \in {\cal N}^+_{s}} e^{\sqrt{2} X_s(u)} g \bigg(\frac{X_s(u)}{L_{T({\bf X}_0)}(s)} \bigg) \rightarrow_p \frac{\sqrt{2}}{\pi} \int_0^1 \frac{\pi}{2} \sin(\pi y) g(y) \: dy.
\end{equation}
Also, we have
\begin{equation}\label{Mprob}
\frac{M({\bf X}_s)}{L_{T({\bf X}_0)}(s)} \rightarrow_p 1
\end{equation}
and
\begin{equation}\label{Nprob}
\frac{\log N({\bf X}_s)}{\sqrt{2} L_{T({\bf X}_0)}(s)} \rightarrow_p 1
\end{equation}
\end{lemma}

\begin{proof}
Our strategy is to apply Lemma \ref{momlem}, conditional on the initial configuration at time zero.
To do this, let $(\eps_n)_{n=1}^{\infty}$ be a sequence of positive numbers tending to zero.  Let
$$\rho_n = \sqrt{2} + \eps_n, \qquad s_n = \eps_n^{-5/6}, \qquad t_n = T(D^{\rho_n}) \wedge \eps_n^{-9/8},$$
where $D^{\rho_n}$ is a random variable having law $\mathcal{D}^{\rho_n}$ which will serve as the initial configuration of particles at time zero.
Note that $\eps_n$ and therefore $s_n$ are deterministic but $t_n$ is random.
We have $s_n/t_n \rightarrow_p 0$ by Lemma~\ref{T0upper}.  By Lemma \ref{biglemma}, we have $$L_{T(D^{\rho_n})} - M(D^{\rho_n}) \rightarrow_p \infty \qquad\mbox{as }n \rightarrow \infty.$$  Also, $\P(M(D^{\rho_n}) > \frac{c}{2}\eps_n^{-3/8}) \rightarrow 0$ as $n \rightarrow \infty$ by Lemma \ref{M0bound}.
Combining these two results, we get $$L_{t_n} - M(D^{\rho_n}) \rightarrow_p \infty \qquad\mbox{as }n \rightarrow \infty.$$
It now follows from Skorokhod's Representation Theorem that we may construct the random variables $D^{\rho_n}$ on one probability space in such a way that almost surely, we have
\begin{equation}\label{snbounds}
\lim_{n \rightarrow \infty} \frac{t_n^{2/3}}{s_n} = 0, \qquad \lim_{n \rightarrow \infty} \frac{s_n}{t_n} = 0,
\end{equation}
and
\begin{equation}\label{LMlimit}
\lim_{n \rightarrow \infty} \big( L_{t_n} - M(D^{\rho_n}) \big) = \infty.
\end{equation}
The results \eqref{snbounds} and \eqref{LMlimit}, combined with the fact that $\eps_n s_n \rightarrow 0$ as $n \rightarrow \infty$, allow us to apply Lemma \ref{momlem}, conditional on the initial configurations $D^{\rho_n}$.

The result \eqref{LMlimit} implies that almost surely $L_{t_n} - M(D^{\rho_n}) > 2$ for sufficiently large $n$.  Then the definition of $T(D^{\rho_n})$ implies that almost surely $Z_{T(D^{\rho_n})}(0) = 1/2$ for sufficiently large $n$.  Because $Z_{t_n}(0) > 1/2$ when $t_n < T(D^{\rho_n})$, it follows that almost surely
\begin{equation}\label{ZTlim}
Z_{t_n}(0) \geq 1/2 \qquad\mbox{for sufficiently large }n.
\end{equation}
It also follows from the definitions of $z(x,L)$ and $y(x,L)$ that when equation \eqref{LMlimit} holds, we have $Y_{t_n}(0)/Z_{t_n}(0) \rightarrow 0$ as $n \rightarrow \infty$.  We also have $s_n/t_n \rightarrow 0$ and $L_{t_n} \rightarrow \infty$ as $n \rightarrow \infty$.  Therefore, the right-hand side of \eqref{mom3} divided by the square of the right-hand side of \eqref{mom1} tends to zero as $n \rightarrow \infty$, and the right-hand side of \eqref{mom4} divided by the square of the right-hand side of \eqref{mom2} tends to zero as $n \rightarrow \infty$.  It follows that we can obtain concentration results by applying the conditional Chebyshev's Inequality.  We get that for the process in which particles are killed not only at the origin but also if they reach $L_{t_n}(r)$ at time $r$, we have as $n \rightarrow \infty$,
\begin{equation}\label{fkilling}
\frac{L_{t_n}^3}{Z_{t_n}(0)} e^{-\sqrt{2} L_{t_n}(s_n)} \sum_{u \in \mathcal{N}^+_{s_n}} f(X_{s_n}(u)) \rightarrow_p \frac{\pi}{\sqrt{2}} \int_0^{\infty} 2ye^{-\sqrt{2} y} f(y) \: dy
\end{equation}
and
\begin{equation}\label{gkilling}
\frac{L_{t_n}}{Z_{t_n}(0)} e^{-\sqrt{2} L_{t_n}(s_n)} \sum_{u \in \mathcal{N}^+_{s_n}} e^{\sqrt{2} X_{s_n}(u)} g \bigg( \frac{X_{s_n}(u)}{L_{t_n}(s_n)}\bigg) \rightarrow_p \frac{2 \sqrt{2}}{\pi} \int_0^1 \frac{\pi}{2} \sin(\pi y) g(y) \: dy.
\end{equation}

Let $0 < \theta < 1$.  By applying \eqref{gkilling} to the function $g(x) = \1_{[\theta, 1]}(x)$, we see that
\begin{equation}\label{MTL}
\lim_{n \rightarrow \infty} \P_{\mathcal{D}^{\rho_n}}^{\rho_n, t_n}(M({\bf X}_{s_n}) > \theta L_{t_n}(s_n)) = 1.
\end{equation}
Because the position of the maximum particle can only become larger if particles are not killed at an upper boundary, it follows that
$$\lim_{n \rightarrow \infty} \P_{\mathcal{D}^{\rho_n}}^{\rho_n}(M({\bf X}_{s_n}) > \theta L_{t_n}(s_n)) = 1.$$
However, Lemma \ref{M0bound} and Corollary \ref{survivalcor} imply that $$\lim_{n \rightarrow \infty} \P_{\mathcal{D}^{\rho_n}}^{\rho_n}(M({\bf X}_{s_n}) > \theta L_{\eps_n^{-9/8}}(s_n)) = 0.$$  It now follows from the definition of $t_n$ that $$\lim_{n \rightarrow \infty} \P(T(D^{\rho_n}) < \eps_n^{-9/8}) = 1.$$
Therefore, $t_n$ can be replaced by $T(D^{\rho_n}) = T({\bf X}_0)$ in \eqref{fkilling} and \eqref{gkilling}.  Also, because almost surely $Z_{T(D^{\rho_n})}(0) = 1/2$ for sufficiently large $n$ as noted above, we can replace $Z_{t_n}(0)$ by $1/2$ in \eqref{fkilling} and \eqref{gkilling}.  It follows that \eqref{probf} and \eqref{probg} both hold for the process in which particles are killed at $L_{t_n}(r)$ at time $r$.  Furthermore, by letting $\theta \rightarrow 1$ in \eqref{MTL} and noting that $M({\bf X}_{s_n}) \leq L_{t_n}(s_n)$ when particles are killed at the upper boundary, the result \eqref{Mprob} also holds for the process in which particles are killed at $L_{t_n}(r)$ at time $r$.

To complete the proof of \eqref{probf}, \eqref{probg}, and (\ref{Mprob}), it remains to show that, with probability tending to one as $n \rightarrow \infty$, no particle reaches $L_{t_n}(r)$ for $r \in [0, s_n]$, which will establish that results for the process with killing at the upper boundary also hold for the process without killing.  However, this follows immediately from Lemma \ref{meanR} because $s_n/t_n \rightarrow 0$ and $Y_{t_n}(0)/Z_{t_n}(0) \rightarrow 0$ almost surely as $n \rightarrow \infty$, and $Z_{t_n}(0) = 1/2$ with probability tending to one as $n \rightarrow \infty$.  Finally, \eqref{Nprob} follows by choosing $f(x) = 1$ for all $x$ in \eqref{probf} and taking logarithms.
\end{proof}

\begin{proof}[Proof of Theorem \ref{etachi}]
Let $f: [0, \infty) \rightarrow [0, \infty)$ and $g: [0, 1] \rightarrow [0, \infty)$ be bounded continuous functions.  Let $s = \eps^{-5/6}$.  Dividing the expression on the left-hand side of \eqref{probf} for a general $f$ by the expression when $f(x) = 1$ for all $x$, we get that under $\P^{\rho}_{\mathcal{D}^{\rho}}$, as $\eps \rightarrow 0$,
$$\frac{1}{N({\bf X}_s)} \sum_{u \in {\cal N}^+_{s}} f(X_s(u)) \rightarrow_p \int_0^{\infty} 2y e^{-\sqrt{2} y} f(y) \: dy.$$
Likewise, dividing the expression on the left-hand side of \eqref{probg} for a general $g$ by the expression when $g(x) = 1$ for all $x$, we get that under ${\cal D}^{\rho_n}$,
$$\bigg( \sum_{u \in {\cal N}^+_{s}} e^{\sqrt{2} X_s(u)} \bigg)^{-1} \sum_{u \in {\cal N}^+_{s}} e^{\sqrt{2} X_s(u)} g \bigg(\frac{X_s(u)}{L_{T({\bf X}_0)}(s)} \bigg) \rightarrow_p \int_0^1 \frac{\pi}{2} \sin(\pi y) g(y) \: dy.$$  By \eqref{Mprob} and the continuity of $g$, it now follows that
$$\bigg( \sum_{u \in {\cal N}^+_{s}} e^{\sqrt{2} X_s(u)} \bigg)^{-1} \sum_{u \in {\cal N}^+_{s}} e^{\sqrt{2} X_s(u)} g \bigg(\frac{X_s(u)}{M({\bf X}_s)} \bigg) \rightarrow_p \int_0^1 \frac{\pi}{2} \sin(\pi y) g(y) \: dy.$$
Because $(\rho^2/2 - 1)s \rightarrow 0$ as $\eps \rightarrow 0$, it now follows from the quasi-stationary established in Corollary \ref{survivalcor} that
\begin{equation}\label{ffinal}
\frac{1}{N({\bf X}_0)} \sum_{u \in {\cal N}^+_{0}} f(X_0(u)) \rightarrow_p \int_0^{\infty} 2y e^{-\sqrt{2} y} f(y) \: dy.
\end{equation}
and
$$\bigg( \sum_{u \in {\cal N}^+_{0}} e^{\sqrt{2} X_0(u)} \bigg)^{-1} \sum_{u \in {\cal N}^+_{0}} e^{\sqrt{2} X_0(u)} g \bigg(\frac{X_0(u)}{M({\bf X}_0)} \bigg) \rightarrow_p \int_0^1 \frac{\pi}{2} \sin(\pi y) g(y) \: dy.$$
According to Theorem 16.16 of \cite{kal02}, these two convergence results imply Theorem \ref{etachi}.
\end{proof}

We now record the following consequence of Theorem \ref{etachi} which will be important in the next section.  For $\nu \in \mathfrak{P}$ given by \eqref{nudef} and $\Delta > 0$, define $$Z^{\Delta}(\nu, L) = \sum_i z(x_i, L)\1_{\{x_i \leq \Delta\}},$$ which is the contribution to $Z(\nu, L)$ from particles located between $0$ and $\Delta$.  The following lemma shows that this contribution is small for a configuration of particles having law $\mathcal{D}^{\rho}$.

\begin{lemma}\label{leftedge}
For all $\Delta > 0$ and all $\theta > 0$, we have
\begin{equation}\label{ZDrho}
\lim_{\eps \rightarrow 0} \P(Z^{\Delta}(D^{\rho}, L_{T(D^{\rho})}) > \theta) = 0.
\end{equation}
\end{lemma}

\begin{proof}
Note that $L_{T(D^{\rho})} \rightarrow_p \infty$ as $\eps \rightarrow 0$ by Lemma \ref{T0upper}.  Also, using that $\lim_{y \rightarrow 0} y^{-1} \sin(y) = 1$, for any fixed $x > 0$ we have $$\lim_{L \rightarrow \infty} \frac{z(x,L)}{\pi \sqrt{2} e^{-\sqrt{2} L} \cdot xe^{\sqrt{2} x}} = 1,$$
and this convergence is uniform over $x$ in compact intervals.  Therefore, if for each positive integer $k$ we define the function $f_k(x) = xe^{\sqrt{2} x} \1_{\{x \leq k \Delta\}}$, then
under $\P^{\rho}_{\mathcal{D}^{\rho}}$, we have as $\eps \rightarrow 0$,
$$\frac{ \pi \sqrt{2} e^{-\sqrt{2} L_{T({\bf X}_0)}}}{Z^{k \Delta}({\bf X}_0, L_{T({\bf X}_0)})} \sum_{u \in \mathcal{N}_0} f_k(X_0(u))  \rightarrow_p 1.$$
Now, by applying \eqref{ffinal} to the function $f_k$ for $k = m$ and $k = 1$, we get that as $\eps \rightarrow 0$,
$$\frac{Z^{\Delta}(D^{\rho}, L_{T(D^{\rho})})}{Z^{m \Delta}(D^{\rho}, L_{T(D^{\rho})})} \rightarrow_p \frac{\int_0^{\infty} 2y e^{-\sqrt{2} y} f_1(y) \: dy}{\int_0^{\infty} 2y e^{-\sqrt{2} y} f_m(y) \: dy} = \frac{\int_0^{\Delta} 2y^2 \: dy}{\int_0^{m \Delta} 2y^2 \: dy} = \frac{1}{m^3},$$
which tends to zero as $m \rightarrow \infty$.
Because $Z(D^{\rho}, L_{T(D^{\rho})}) \leq 1/2$ by the definition of $T(D^{\rho})$, and therefore $Z^{m \Delta}(D^{\rho}, L_{T(D^{\rho})}) \leq 1/2$ for all $m$, the result follows.
\end{proof}

\subsection{Proof of Theorem \ref{NandM}}
\label{subsec:NandM}

So that we can take advantage of the coupling presented in section \ref{couplesec}, we will begin by proving that the values of $T(D^{\rho})$ and $T(D^{\rho} - \Delta)$ are approximately the same.

\begin{lemma}\label{TTDelta}
Let $\Delta > 0$.  Then there exists a positive constant $k_{\Delta}$, depending on $\Delta$, such that $$\lim_{\eps \rightarrow 0} \P\big(T(D^{\rho})- k_{\Delta} T(D^{\rho})^{2/3} \leq T(D^{\rho} - \Delta) \leq T(D^{\rho}) \big) = 1.$$
\end{lemma}

\begin{proof}
Because $Z(D^{\rho}, L_{T(D^{\rho})}) = 1/2$ as long as $M(D^{\rho}) < L_{T(D^{\rho})} - 2$, it follows from Lemma~\ref{biglemma} that
$$\lim_{\eps \rightarrow 0} \P(Z(D^{\rho}, L_{T(D^{\rho})}) = 1/2) = 1.$$
We now consider what happens when we change the configuration of particles from $D^{\rho}$ to $D^{\rho} - \Delta$.  It follows from the upper bound in (\ref{dzx}) that for $x \in [\Delta + 1, L_{T(D^{\rho})} - 2]$, we have
$$z(x - \Delta, L_{T(D^{\rho})}) \geq e^{-\Delta(\sqrt{2} + \pi/2)} z(x, L_{T(D^{\rho})}).$$
Lemma \ref{leftedge} with $\theta = 1/4$ implies that the contribution to $Z(D^{\rho}, L_{T(D^{\rho})})$ from all particles below $\Delta + 1$ is at most $1/4$ with probability tending to one as $\eps \rightarrow 0$.
Therefore, $$\lim_{\eps \rightarrow 0} \P\bigg(Z(D^{\rho} - \Delta, L_{T(D^{\rho})}) \geq \frac{1}{4} e^{-\Delta(\sqrt{2} + \pi/2)} \bigg) = 1.$$

By (\ref{dLx}), if $0 < x \leq M(D^{\rho})$ and $M(D^{\rho}) + 2 \leq L \leq L_{T(D^{\rho})}$, then
$$z(x, L) \geq z(x, L_{T(D^{\rho})}) e^{(L_{T(D^{\rho})} - L)/8}.$$
Therefore, on the event that $Z(D^{\rho} - \Delta, L_{T(D^{\rho})}) \geq \frac{1}{4} e^{-\Delta(\sqrt{2} + \pi/2)}$, we have $Z(D^{\rho} - \Delta, L) \geq 1/2$ if
$L \leq L_{T(D^{\rho})}- a_{\Delta}$, where $$a_{\Delta} = 8 \bigg( \Delta \Big(\sqrt{2} + \frac{\pi}{2} \Big) + \log 2 \bigg).$$
Thus, with probability tending to one as $\eps \rightarrow 0$, we have
\begin{equation}\label{LTaDel}
L_{T(D^{\rho})} - a_{\Delta} \leq L_{T(D^{\rho} - \Delta)} \leq L_{T(D^{\rho})},
\end{equation}
and therefore
$$\bigg(T(D^{\rho})^{1/3} - \frac{a_{\Delta}}{c} \bigg)^3 \leq T(D^{\rho} - \Delta) \leq T(D^{\rho}).$$
The result follows, with $k_{\Delta} = 3a_{\Delta}/c$.
\end{proof}

The next result, which is an extension of Lemma \ref{lemma2.11}, allows us to relate the extinction time $\zeta$ for the process to $T({\bf X}_0)$, which is a function of the initial configuration of particles.  This result will allow us to use what is known from Corollary \ref{survivalcor} about the distribution of $\zeta$ to obtain asymptotics for $M(D^{\rho})$ and $N(D^{\rho})$.

\begin{lemma}\label{zetaTlemma}
Let $\delta > 0$.  There exists a positive constant $k_0$ such that for sufficiently small $\eps$, we have
\begin{equation}
\P_{\mathcal{D}^{\rho}}^{\rho} \big( |\zeta - T({\bf X}_0)| \leq k_0 T({\bf X}_0)^{2/3} \big) > 1 - \delta.
\end{equation}
\end{lemma}

\begin{proof}
We will use the coupling described in section \ref{couplesec}.
Let $\delta > 0$, and choose $k$, $t_0$, and $a$ as in Lemma \ref{lemma2.11} but with $\delta/4$ in place of $\delta$.  By Corollary \ref{survivalcor}, we can choose a constant $\Delta > 0$ large enough that $$\P_{\mathcal{D}^{\rho}}^{\rho}(\zeta > \Delta \eps^{-1}) < \delta/4$$ for sufficiently small $\eps$.  Recall the constants $a_{\Delta}$ and $k_{\Delta}$ from Lemma \ref{TTDelta}.  Choose a positive constant $t_1 > t_0$ large enough that $t_1 - k_{\Delta}t_1^{2/3} \geq t_0$.
By Lemma \ref{T0upper},
\begin{equation}\label{Drho1}
\lim_{\eps \rightarrow 0} \P(T(D^{\rho}) > t_1) = 1.
\end{equation}
By Lemma \ref{TTDelta}, it follows that
\begin{equation}\label{Drho1Del}
\lim_{\eps \rightarrow 0} \P(T(D^{\rho} - \Delta) > t_0) = 1.
\end{equation}
It follows from Lemma \ref{biglemma} that
\begin{equation}\label{Drho2}
\lim_{\eps \rightarrow 0} \P(L_{T(D^{\rho})} - M(D^{\rho}) \geq a + a_{\Delta}) = 1,
\end{equation}
and then, because $M(D^{\rho} - \Delta) \leq M(D^{\rho})$, equation \eqref{LTaDel} yields
\begin{equation}\label{Drho2Del}
\lim_{\eps \rightarrow 0} \P(L_{T(D^{\rho} - \Delta)} - M(D^{\rho} - \Delta) \geq a) = 1.
\end{equation}
In view of (\ref{Drho1}), (\ref{Drho1Del}), (\ref{Drho2}), and (\ref{Drho2Del}), we can apply Lemma \ref{lemma2.11} with $D^{\rho}$ and with $D^{\rho} - \Delta$ in place of $\nu$ to get that for sufficiently small $\eps$,
\begin{equation}\label{Tbound1}
\P^{\rho}_{\mathcal{D}^{\rho}}\big(|\zeta^U - T({\bf X}_0^U)| \leq k T({\bf X}_0^U)^{2/3} \big) > 1 - \delta/4
\end{equation}
and
\begin{equation}\label{Tbound2}
\P^{\rho}_{\mathcal{D}^{\rho}}\big(|\zeta^{L, \Delta} - T({\bf X}_0^{L, \Delta})| \leq k T({\bf X}_0^{L, \Delta})^{2/3} \big) > 1 - \delta/4.
\end{equation}
Because $T({\bf X}_0^U) = T({\bf X}_0)$ and $\zeta \leq \zeta^U$ by \eqref{Uextinct}, equation \eqref{Tbound1} implies
\begin{equation}\label{zetabound1}
\P_{\mathcal{D}^{\rho}}^{\rho}(\zeta \leq T({\bf X}_0) + kT({\bf X}_0)^{2/3}) > 1 - \delta/4.
\end{equation}
Also, using \eqref{Uextinct} and (\ref{Tbound2}),
\begin{equation}\label{zetabound2}
\P_{\mathcal{D}^{\rho}}^{\rho}\big(\zeta \leq T({\bf X}_0^{L, \Delta}) - k T({\bf X}_0^{L, \Delta})^{2/3}\big) \leq \P_{\mathcal{D}^{\rho}}^{\rho}(\zeta > \Delta \eps^{-1}) + \delta/4 < \delta/2.
\end{equation}
Combining (\ref{zetabound1}) and (\ref{zetabound2}) with Lemma \ref{TTDelta}, we get that for sufficiently small $\eps$,
$$\P_{\mathcal{D}^{\rho}}^{\rho}\big( T({\bf X}_0) - (k + k_{\Delta}) T({\bf X}_0)^{2/3} \leq \zeta \leq T({\bf X}_0) + kT({\bf X}_0)^{2/3}\big) > 1 - \delta.$$
The result now follows with $k_0 = k + k_{\Delta}$.
\end{proof}

\begin{proof}[Proof of Theorem \ref{NandM}]
Write $s = \eps^{-5/6}$.  We now compare $T({\bf X}_0)$ and $T({\bf X}_s)$.  Let $\delta > 0$.  By Lemma \ref{zetaTlemma}, for sufficiently small $\eps$,
\begin{equation}\label{zetaT0}
\P_{\mathcal{D}^{\rho}}^{\rho}\big(|\zeta - T({\bf X}_0)| \leq k_0 T({\bf X}_0)^{2/3}\big) > 1 - \delta.
\end{equation}
Corollary \ref{survivalcor} applied to the function $F(\nu) = \P_{\nu}^{\rho}(|\zeta - T(\nu)| \leq k_0 T(\nu)^{2/3})$, combined with \eqref{zetaT0} gives that for sufficiently small $\eps$,
\begin{equation}\label{zetaTscond}
\P_{\mathcal{D}^{\rho}}^{\rho}\big(|(\zeta - s) - T({\bf X}_s)| \leq k_0 T({\bf X}_s)^{2/3} \,|\, \zeta > s\big) > 1 - \delta.
\end{equation}
It also follows from Corollary \ref{survivalcor} that
\begin{equation}\label{earlydeath}
\lim_{\eps \rightarrow 0} \P^{\rho}_{{\cal D}^{\rho}}(\zeta > s) = 1,
\end{equation}
so \eqref{zetaTscond} implies that for sufficiently small $\eps$,
\begin{equation}\label{zetaTs}
\P_{\mathcal{D}^{\rho}}^{\rho}\big(|(\zeta - s) - T({\bf X}_s)| \leq k_0 T({\bf X}_s)^{2/3}\big) > 1 - 2 \delta.
\end{equation}
Therefore, combining (\ref{zetaT0}) and (\ref{zetaTs}) gives that for sufficiently small $\eps$,
$$\P_{\mathcal{D}^{\rho}}^{\rho}\big( |T({\bf X}_0) - T({\bf X}_s)| \leq s + k_0 T({\bf X}_0)^{2/3} + k_0 T({\bf X}_s)^{2/3} \big) > 1 - 3 \delta.$$
It now follows from Lemma \ref{T0upper} that under $\P^{\rho}_{{\cal D}^{\rho}}$, we have
$T({\bf X}_0)/T({\bf X}_s) \rightarrow_p 1$ as $\eps \rightarrow 0$, and therefore
$$\frac{L_{T({\bf X}_0)}(s)}{L_{T({\bf X}_s)}} = \bigg( \frac{T({\bf X}_0) - s}{T({\bf X}_0)} \cdot \frac{T({\bf X}_0)}{T({\bf X}_s)} \bigg)^{1/3} \rightarrow_p 1.$$
Therefore, the equations \eqref{Mprob} and \eqref{Nprob} imply
$$\frac{\log N({\bf X}_s)}{L_{T({\bf X}_s)}} \rightarrow_p \sqrt{2}, \qquad \frac{M({\bf X}_s)}{L_{T({\bf X}_s)}} \rightarrow_p 1.$$
It now follows from \eqref{earlydeath} and the quasi-stationarity established in Corollary \ref{survivalcor} that under $\P^{\rho}_{{\cal D}^{\rho}}$,
$$\frac{\log N({\bf X}_0)}{L_{T({\bf X}_0)}} \rightarrow_p \sqrt{2}, \qquad \frac{M({\bf X}_0)}{L_{T({\bf X}_0)}} \rightarrow_p 1.$$
Also, under $\P^{\rho}_{{\cal D}^{\rho}}$, we have $\zeta/T({\bf X}_0) \rightarrow_p 1$ as $\eps \rightarrow 0$ by Lemma \ref{T0upper} and Lemma \ref{zetaTlemma}, so
$$\frac{\log N({\bf X}_0)}{c \zeta^{1/3}} \rightarrow_p \sqrt{2}, \qquad \frac{M({\bf X}_0)}{c \zeta^{1/3}} \rightarrow_p 1.$$
Corollary \ref{survivalcor} implies that $\eps \zeta \Rightarrow V/\sqrt{2}$ as $\eps \rightarrow 0$ under ${\cal D}^{\rho}$, so
we have the joint convergence
$$\bigg( \eps \zeta, \eps^{1/3} \log N({\bf X}_0), \eps^{1/3} M({\bf X}_0) \bigg) \Rightarrow \bigg( \frac{1}{\sqrt{2}} V, \frac{\sqrt{2} c}{2^{1/6}} V^{1/3}, \frac{c}{2^{1/6}} V^{1/3} \bigg),$$
which is equivalent to the statement of Theorem \ref{NandM}.
\end{proof}

\subsection{Moment Calculations}\label{momproof}

In this section, we prove Lemma \ref{momlem}.  Recall that we are considering the process in which particles at time $s$ are killed when they reach either $0$ or $L_t(s)$.  The calculations in this section are similar to those in section 3.4 in \cite{bbs15} but use ideas from \cite{ms22}.

\begin{proof}[Proof of \eqref{mom1}]
Using Lemma \ref{densitybound}, we have
\begin{align*}
\E^{\rho, t}_x \bigg[ \sum_{u \in {\cal N}^+_s} f(X_s(u)) \bigg] &= \int_0^{L_t(s)} q_s^{\rho, t}(x,y) f(y) \: dy \\
&\sim \int_0^{L_t(s)} \frac{2}{L_t} e^{\sqrt{2} (L_t(s) - L_t)} e^{\sqrt{2} x} \sin \bigg( \frac{\pi x}{L_t} \bigg) e^{-\sqrt{2} y} \sin \bigg( \frac{\pi y}{L_t(s)} \bigg) f(y) \: dy.
\end{align*}
Because $L_t(s) \rightarrow \infty$ as $t \rightarrow \infty$ and $\sin(z) \sim z$ as $z \rightarrow 0$, we have
\begin{equation}\label{intesin}
\int_0^{L_t(s)} e^{-\sqrt{2} y} \sin \bigg( \frac{\pi y}{L_t(s)} \bigg) f(y) dy \sim \frac{\pi}{L_t(s)} \int_0^{\infty} y e^{-\sqrt{2} y} f(y) \: dy.
\end{equation}
Noting also that $L_t \sim L_t(s)$ because $s \ll t$, and recalling \eqref{zdef}, it follows that
\begin{align}\label{expf}
\E^{\rho, t}_x \bigg[ \sum_{u \in {\cal N}^+_s} f(X_s(u)) \bigg] &\sim \frac{\pi}{L_t^2} e^{\sqrt{2} (L_t(s) - L_t)} e^{\sqrt{2} x} \sin \bigg( \frac{\pi x}{L_t} \bigg) \int_0^{\infty} 2y e^{-\sqrt{2} y} f(y) \: dy \nonumber \\
&= \frac{\pi}{\sqrt{2} L_t^3} e^{\sqrt{2} L_t(s)} z(x, L_t) \int_0^{\infty} 2y e^{-\sqrt{2} y} f(y) \: dy.
\end{align}
Summing over all particles in $\mathcal{N}_0$ now gives the result.
\end{proof}

\begin{proof}[Proof of \eqref{mom2}]
Using Lemma \ref{densitybound} and then recalling (\ref{zdef}), we have
\begin{align*}
&\E^{\rho, t}_x \bigg[ \sum_{u \in {\cal N}^+_s} e^{\sqrt{2} X_s(u)} g \bigg(\frac{X_s(u)}{L_t(s)} \bigg) \bigg] \\
&\qquad = \int_0^{L_t(s)} q_s^{\rho, t}(x,y) e^{\sqrt{2} y} g \bigg( \frac{y}{L_t(s)} \bigg) \: dy \\
&\qquad \sim \int_0^{L_t(s)} \frac{2}{L_t} e^{\sqrt{2} (L_t(s) - L_t)} e^{\sqrt{2} x} \sin \bigg( \frac{\pi x}{L_t} \bigg) \sin \bigg( \frac{\pi y}{L_t(s)} \bigg) g \bigg( \frac{y}{L_t(s)} \bigg) \: dy \\
&\qquad = \frac{\sqrt{2}}{L_t^2} e^{\sqrt{2} L_t(s)} z(x, L_t) \int_0^{L_t(s)} \sin \bigg( \frac{\pi y}{L_t(s)} \bigg) g \bigg( \frac{y}{L_t(s)} \bigg) \: dy.
\end{align*}
Making the substitution $v = y/L_t(s)$, we get
$$\E^{\rho, t}_x \bigg[ \sum_{u \in {\cal N}^+_s} e^{\sqrt{2} X_s(u)} g \bigg(\frac{X_s(u)}{L_t(s)} \bigg) \bigg] \sim \frac{\sqrt{2} L_t(s)}{L_t^2} e^{\sqrt{2} L_t(s)} z(x, L_t) \int_0^1 \sin(\pi y) g(y) \: dy.$$
We can now obtain \eqref{mom2} by recalling that $L_t \sim L_t(s)$ and then summing over all particles in $\mathcal{N}_0$.
\end{proof}

\begin{proof}[Proof of \eqref{mom3}]
Standard second moment calculations, which go back to \cite{inw69}, give
\begin{align*}
\Var^{\rho, t}_x \bigg( \sum_{u \in {\cal N}^+_s} f(X_s(u)) \bigg) &\leq \E^{\rho, t}_x \bigg[ \bigg( \sum_{u \in {\cal N}^+_s} f(X_s(u)) \bigg)^2 \bigg] \\
&= \int_0^{L_t(s)} q^{\rho, t}_s(x,y) f(y)^2 \: dy \\
&\qquad + 2 \int_0^s \int_0^{L_t(r)} q^{\rho, t}_r(x,z) \bigg( \int_0^{L_t(s)} q^{\rho, t}_{r,s}(z,y) f(y) \: dy \bigg)^2 \: dz \: dr.
\end{align*}
Because $f$ is bounded, it follows that
\begin{align}\label{var2terms}
&\Var^{\rho, t}_x \bigg( \sum_{u \in {\cal N}^+_s} f(X_s(u)) \bigg) \nonumber \\
&\qquad \lesssim \int_0^{L_t(s)} q^{\rho, t}_s(x,y) \: dy + \int_0^s \int_0^{L_t(r)} q^{\rho, t}_r(x,z) \bigg( \int_0^{L_t(s)} q^{\rho, t}_{r,s}(z,y) \: dy \bigg)^2 \: dz \: dr.
\end{align}
The first term is the same as the expectation in \eqref{expf} when $f(y) = 1$ for all $y$, so
\begin{equation}\label{term1bound}
\int_0^{L_t(s)} q_s(x,y) \: dy \lesssim \frac{e^{\sqrt{2} L_t(s)}}{L_t^3} z(x, L_t).
\end{equation}

The second term is more involved, and we will separately consider the cases $0 < r \leq s - t^{2/3}$ and $s - t^{2/3} < r < s$.  Suppose first that $0 \leq r \leq s - t^{2/3}$.  Note that our assumptions imply that $L_t(r) \asymp L_t$ and that our assumptions, along with \eqref{qGirsanov}, imply that $q_r^{\rho, t}(x,z) \lesssim q_r^{\sqrt{2}, t}(x,z)$.  Therefore, using \eqref{qrs}, we get
$$q_r^{\rho, t}(x,z) \lesssim \frac{1}{L_t} e^{\sqrt{2}(x - z)} w_{\tau_t(0,r)} \bigg( \frac{x}{L_t}, \frac{z}{L_t(r)} \bigg).$$
Also, recalling the definition of $\tau_t(r,s)$ from \eqref{taudef}, we have $\tau_t(r,s) \gtrsim 1$.  Therefore, using \eqref{wseq}, \eqref{dbound}, \eqref{taudef}, \eqref{qrs}, and \eqref{qGirsanov}, and using \eqref{intesin} for the last inequality, we get
\begin{align}\label{innerint}
\int_0^{L_t(s)} q_{r,s}^{\rho, t}(z,y) \: dy &\lesssim \int_0^{L_t(s)} \frac{1}{L_t} e^{\sqrt{2}(z - y)} e^{-\pi^2 \tau_t(r,s)/2} \sin \bigg( \frac{\pi z}{L_t(r)} \bigg) \sin \bigg( \frac{\pi y}{L_t(s)} \bigg) \: dy \nonumber \\
&= \frac{1}{L_t} e^{\sqrt{2} z} \sin \bigg( \frac{\pi z}{L_t(r)} \bigg) e^{\sqrt{2}(L_t(s) - L_t(r))} \int_0^{L_t(s)} e^{-\sqrt{2} y} \sin \bigg( \frac{\pi y}{L_t(s)} \bigg) \: dy \nonumber \\
&\lesssim \frac{1}{L_t^2} e^{\sqrt{2} z} \sin \bigg( \frac{\pi z}{L_t(r)} \bigg) e^{\sqrt{2}(L_t(s) - L_t(r))}.
\end{align}
Therefore,
\begin{align*}
&\int_0^{s - t^{2/3}} \int_0^{L_t(r)} q^{\rho, t}_r(x,z) \bigg( \int_0^{L_t(s)} q^{\rho, t}_{r,s}(z,y) \: dy \bigg)^2 \: dz \: dr \\
&\quad\lesssim \frac{e^{\sqrt{2} x} e^{2 \sqrt{2} L_t(s)}}{L_t^5} \int_0^{s - t^{2/3}} e^{-2 \sqrt{2} L_t(r)} \int_0^{L_t(r)} e^{\sqrt{2} z} \sin \bigg( \frac{\pi z}{L_t(r)} \bigg)^2 w_{\tau_t(0,r)} \bigg( \frac{x}{L_t}, \frac{z}{L_t(r)} \bigg) \: dz \: dr.
\end{align*}
Making the substitution $v = L_t(r) - z$ and using the bound $\sin(\pi(L_t(r) - v)/L_t(r)) \leq \pi v/L_t(r)$, we get
\begin{align*}
&\int_0^{s - t^{2/3}} \int_0^{L_t(r)} q^{\rho, t}_r(x,z) \bigg( \int_0^{L_t(s)} q^{\rho, t}_{r,s}(z,y) \: dy \bigg)^2 \: dz \: dr \\
&\quad\lesssim \frac{e^{\sqrt{2} x} e^{2 \sqrt{2} L_t(s)}}{L_t^5} \int_0^{s - t^{2/3}} \frac{e^{-\sqrt{2} L_t(r)}}{L_t(r)^2} \int_0^{L_t(r)} e^{-\sqrt{2} v} v^2 w_{\tau_t(0,r)} \bigg( \frac{x}{L_t}, \frac{L_t(r) - v}{L_t(r)} \bigg) \: dv \: dr \\
&\quad= \frac{e^{\sqrt{2} (x - L_t)} e^{2 \sqrt{2} L_t(s)}}{L_t^5} \int_0^{s - t^{2/3}} \frac{e^{\pi^2\tau_t(0,r)/2}}{L_t(r)^2} \int_0^{L_t(r)} e^{-\sqrt{2} v} v^2 w_{\tau_t(0,r)} \bigg( \frac{x}{L_t}, \frac{L_t(r) - v}{L_t(r)} \bigg) \: dv \: dr. \\
\end{align*}
Because $w_s(x,y) = w_s(1-x, 1-y)$ by symmetry, this expression matches, up to a scaling by $\sqrt{2}$ and up to a factor of $e^{2\sqrt{2} L_t(s)}/L_t^6$, with equation (5.33) in \cite{ms22} in the case when $A = 0$, $r = 0$ and $s=s-t^{2/3}$.  Equation (5.33) in \cite{ms22} is bounded by the sum of the expressions in (5.35) and (5.36) in \cite{ms22}.  Therefore, by following the same reasoning as on pages 967-968 of \cite{ms22} or by simply applying that result, we obtain
\begin{align}\label{term2Abound}
&\int_0^{s - t^{2/3}} \int_0^{L_t(r)} q^{\rho, t}_r(x,z) \bigg( \int_0^{L_t(s)} q^{\rho, t}_{r,s}(z,y) \: dy \bigg)^2 \: dz \: dr \nonumber \\
&\qquad \qquad \lesssim \frac{e^{2 \sqrt{2} L_t(s)}}{L_t^6} \bigg( \frac{\tau_t(0, s - t^{2/3})}{L_t} z(x, L_t) + y(x, L_t) \bigg) \nonumber \\
&\qquad \qquad \lesssim \frac{e^{2 \sqrt{2} L_t(s)}}{L_t^6} \bigg( \frac{s}{t} z(x, L_t) + y(x, L_t) \bigg).
\end{align}

It remains to consider the case when $s - t^{2/3} < r < s$.  Using Lemma \ref{densitybound} to bound $q^{\rho, t}_r(x,z)$ and using \eqref{qrs} to bound $q^{\rho, t}_{r,s}(z,y)$, and then using that $L_t(r) - L_t(s) \lesssim 1$, we get
\begin{align}\label{term2Bprelim}
&\int_{s - t^{2/3}}^s \int_0^{L_t(r)} q^{\rho, t}_r(x,z) \bigg( \int_0^{L_t(s)} q^{\rho, t}_{r,s}(z,y) \: dy \bigg)^2 \: dz \nonumber \\
&\qquad \lesssim \int_{s - t^{2/3}}^s \int_0^{L_t(r)} \frac{1}{L_t} e^{\sqrt{2} (L_t(r) - L_t)} e^{\sqrt{2} x} \sin \bigg( \frac{\pi x}{L_t} \bigg) e^{-\sqrt{2} z} \sin \bigg( \frac{\pi z}{L_t(r)} \bigg) \nonumber \\
&\qquad \qquad \times \bigg( \int_0^{L_t(s)} \frac{1}{L_t(r)} e^{\sqrt{2}(z - y)} w_{\tau_t(r,s)} \bigg( \frac{z}{L_t(r)}, \frac{y}{L_t(s)} \bigg) \: dy \bigg)^2 \: dz \: dr \nonumber \\
&\qquad \lesssim \frac{e^{\sqrt{2} L_t(s)}}{L_t^2} z(x, L_t) \int_{s - t^{2/3}}^s \frac{1}{L_t(r)^2} \nonumber \\
&\qquad \qquad \times \int_0^{L_t(r)} e^{\sqrt{2} z} \sin \bigg( \frac{\pi z}{L_t(r)} \bigg)\bigg( \int_0^{L_t(s)} e^{-\sqrt{2} y} w_{\tau_t(r,s)} \bigg( \frac{z}{L_t(r)}, \frac{y}{L_t(s)} \bigg) \: dy \bigg)^2 \: dz \: dr.
\end{align}
We now divide the integral on the last line of \eqref{term2Bprelim} into three pieces.  Because $e^{-\sqrt{2} y} \leq 1$ and $\int_0^1 w_s(x,y) \: dy \leq 1$, the inner integral in \eqref{term2Bprelim} is bounded above by $L_t(s)$.  Therefore,
\begin{align}\label{2Ba}
&\int_0^{\frac{1}{2} L_t(r)} e^{\sqrt{2} z} \sin \bigg( \frac{\pi z}{L_t(r)} \bigg)\bigg( \int_0^{L_t(s)} e^{-\sqrt{2} y} w_{\tau_t(r,s)} \bigg( \frac{z}{L_t(r)}, \frac{y}{L_t(s)} \bigg) \: dy \bigg)^2 \: dz \nonumber \\
&\qquad \leq L_t(s)^2 \int_0^{\frac{1}{2} L_t(r)} e^{\sqrt{2} z} \sin \bigg( \frac{\pi z}{L_t(r)} \bigg) \: dr \nonumber \\
&\qquad \leq e^{\frac{1}{2} \sqrt{2} L_t(r)} L_t(r) L_t(s)^2 \nonumber \\
&\qquad \lesssim e^{\frac{1}{2} \sqrt{2} L_t(s)} L_t^3.
\end{align}
Likewise,
\begin{align}\label{2Bb}
&\int_{\frac{1}{2} L_t(r)}^{L_t(r)} e^{\sqrt{2} z} \sin \bigg( \frac{\pi z}{L_t(r)} \bigg)\bigg( \int_{\frac{1}{4}L_t(s)}^{L_t(s)} e^{-\sqrt{2} y} w_{\tau_t(r,s)} \bigg( \frac{z}{L_t(r)}, \frac{y}{L_t(s)} \bigg) \: dy \bigg)^2 \: dz \nonumber \\
&\qquad \leq e^{-\frac{1}{2} \sqrt{2} L_t(s)} L_t(s)^2 \int_{\frac{1}{2} L_t(r)}^{L_t(r)} e^{\sqrt{2} z} \sin \bigg( \frac{\pi z}{L_t(r)} \bigg) \: dz \nonumber \\
&\qquad \leq e^{-\frac{1}{2} \sqrt{2} L_t(s)} L_t(s)^2 e^{\sqrt{2} L_t(r)} L_t(r) \nonumber \\
&\qquad \lesssim e^{\frac{1}{2} \sqrt{2} L_t(s)} L_t^3.
\end{align}
It remains to consider the case when $z > \frac{1}{2} L_t(r)$ and $y < \frac{1}{4} L_t(s)$.  We will need a sharper bound on $w_s(x,y)$.  We can bound $w_s(x,y)$ by the density for Brownian motion killed only at zero, which, as in (5.10) in \cite{ms22}, leads to $$w_s(x,y) \lesssim \frac{xy}{s^{3/2}} e^{-(x-y)^2/2s}.$$  Also, using an idea from the proof of Lemma 5.1 in \cite{ms22} and noting that, for all $z \in (0,1)$, either $|x-z|$ or $|y-z|$ must be at least as large as $|x-y|/2$, we have
\begin{align*}
w_s(x,y) &= \int_0^1 w_{s/2}(x,z) w_{s/2}(z,y) \: dz \\
&= \int_0^1 w_{s/2}(1-x, 1-z) w_{s/2}(z,y) \: dz \\
&\lesssim \int_0^1 \frac{(1-x)(1-z)}{s^{3/2}} e^{-(x-z)^2/s} \cdot \frac{zy}{s^{3/2}} e^{-(y-z)^2/s} \: dz \\
&\lesssim \int_0^1 \frac{(1-x)y}{s^3} e^{-(x-y)^2/4s} \: dz \\
&= \frac{(1-x)y}{s^3} e^{-(x-y)^2/4s}.
\end{align*}
Therefore,
\begin{align}\label{2Bc}
&\int_{\frac{1}{2} L_t(r)}^{L_t(r)} e^{\sqrt{2} z} \sin \bigg( \frac{\pi z}{L_t(r)} \bigg)\bigg( \int_0^{\frac{1}{4}L_t(s)} e^{-\sqrt{2} y} w_{\tau_t(r,s)} \bigg( \frac{z}{L_t(r)}, \frac{y}{L_t(s)} \bigg) \: dy \bigg)^2 \: dz \nonumber \\
&\qquad \lesssim \int_{\frac{1}{2} L_t(r)}^{L_t(r)} e^{\sqrt{2} z} \sin \bigg( \frac{\pi z}{L_t(r)} \bigg)\bigg( \int_0^{\frac{1}{4}L_t(s)} e^{-\sqrt{2} y} \frac{(L_t(r) - z) y}{L_t(r) L_t(s) \tau_t(r,s)^3} e^{-1/(64 \tau_t(r,s))} \: dy \bigg)^2 \: dz \nonumber \\
&\qquad \lesssim \frac{e^{\sqrt{2} L_t(r)}}{L_t(r)^3 L_t(s)^2} \cdot \frac{1}{\tau_t(r,s)^6} e^{-1/(32 \tau_t(r,s))} \nonumber \\
&\qquad \lesssim \frac{e^{\sqrt{2} L_t(s)}}{L_t^5} \cdot \frac{1}{\tau_t(r,s)^6} e^{-1/(32 \tau_t(r,s))}.
\end{align}
Therefore, using that $(x + y)^2 \leq 2(x^2 + y^2)$ and plugging the results of \eqref{2Ba}, \eqref{2Bb}, and \eqref{2Bc} into \eqref{term2Bprelim},
\begin{align*}
&\int_{s - t^{2/3}}^s \int_0^{L_t(r)} q^{\rho, t}_r(x,z) \bigg( \int_0^{L_t(s)} q^{\rho, t}_{r,s}(z,y) \: dy \bigg)^2 \: dz \\
&\quad \lesssim \frac{e^{\sqrt{2} L_t(s)}}{L_t^2} z(x, L_t) \int_{s - t^{2/3}}^s \frac{1}{L_t(r)^2} \bigg( e^{\frac{1}{2} \sqrt{2} L_t(s)} L_t^3 + \frac{e^{\sqrt{2} L_t(s)}}{L_t^5} \cdot \frac{1}{\tau_t(r,s)^6} e^{-1/(32 \tau_t(r,s))} \bigg) \: dr.
\end{align*}
Now make the substitution $u = \tau_t(r,s)$, so that $du/dr = - L_t(r)^{-2}$.  Because $\tau_t(s - t^{2/3}, s)\lesssim 1$, we have
\begin{align}\label{term2Bbound}
&\int_{s - t^{2/3}}^s \int_0^{L_t(r)} q^{\rho, t}_r(x,z) \bigg( \int_0^{L_t(s)} q^{\rho, t}_{r,s}(z,y) \: dy \bigg)^2 \: dz \nonumber \\
&\quad \lesssim \frac{e^{\sqrt{2} L_t(s)}}{L_t^2} z(x, L_t) \int_0^{\tau_t(s - t^{2/3}, s)} \bigg( e^{\frac{1}{2} \sqrt{2} L_t(s)} L_t^3 + \frac{e^{\sqrt{2} L_t(s)}}{L_t^5} \cdot \frac{1}{u^6} e^{-1/(32 u)} \bigg) \: du \nonumber \\
&\quad \lesssim \frac{e^{\sqrt{2} L_t(s)}}{L_t^2} z(x, L_t) \bigg( e^{\frac{1}{2} \sqrt{2} L_t(s)} L_t^3 + \frac{e^{\sqrt{2} L_t(s)}}{L_t^5} \int_0^{\infty} \frac{1}{u^6} e^{-1/(32u)} \: du \bigg). \nonumber \\
&\quad \lesssim \frac{e^{2 \sqrt{2} L_t(s)}}{L_t^7} z(x, L_t).
\end{align}
From \eqref{var2terms}, \eqref{term1bound}, \eqref{term2Abound}, and \eqref{term2Bbound}, we get
$$\Var^{\rho, t}_x \bigg( \sum_{u \in {\cal N}^+_s} f(X_s(u)) \bigg) \lesssim \frac{e^{2 \sqrt{2} L_t(s)}}{L_t^6} \bigg( \frac{s}{t} z(x, L_t) + \frac{1}{L_t} z(x, L_t) + y(x, L_t) \bigg).$$  The result \eqref{mom3} now follows by summing the contributions from all particles in ${\cal N}_0$.
\end{proof}

\begin{proof}[Proof of \eqref{mom4}]
Reasoning as in the beginning of the proof of \eqref{mom3}, we have
\begin{align}\label{varg2}
&\Var_x^{\rho, t} \bigg( \sum_{u \in {\cal N}^+_s} e^{\sqrt{2} X_s(u)} g \bigg( \frac{X_s(u)}{L_t(s)} \bigg) \bigg) \nonumber \\
&\qquad \lesssim \int_0^{L_t(s)} e^{2 \sqrt{2} y} q_s^{\rho, t}(x,y) \: dy + \int_0^s \int_0^{L_t(r)} q_r^{\rho, t}(x,z) \bigg( \int_0^{L_t(s)} e^{\sqrt{2} y} q_{r,s}^{\rho, t}(z,y) \: dy \bigg)^2 \: dz \: dr.
\end{align}
To bound the first term, we use Lemma \ref{densitybound} to get
\begin{align}\label{term1g}
\int_0^{L_t(s)} e^{2 \sqrt{2} y} q_s^{\rho, t}(x,y) \: dy &\lesssim \frac{1}{L_t} e^{\sqrt{2}(L_t(s) - L_t)} e^{\sqrt{2} x} \sin \bigg( \frac{\pi x}{L_t} \bigg) \int_0^{L_t(s)} e^{\sqrt{2}y} \sin \bigg( \frac{\pi y}{L_t(s)} \bigg) \: dy \nonumber \\
&\lesssim \frac{e^{\sqrt{2} L_t(s)}}{L_t^2} z(x, L_t) \int_0^{L_t(s)} e^{\sqrt{2} y} \sin \bigg( \frac{\pi y}{L_t(s)} \bigg) \: dy \nonumber \\
&\lesssim \frac{e^{2 \sqrt{2} L_t(s)}}{L_t^3} z(x, L_t).
\end{align}
As in the proof of \eqref{mom3}, we will break the second term into two pieces.  Consider first the case when $0 < r \leq s - t^{2/3}$.  Reasoning as in \eqref{innerint}, we get
\begin{align*}
\int_0^{L_t(s)} e^{\sqrt{2} y} q_{r,s}^{\rho, t}(z,y) \: dy &\lesssim \frac{1}{L_t} e^{\sqrt{2} z} \sin \bigg( \frac{\pi z}{L_t(r)} \bigg) e^{\sqrt{2}(L_t(s) - L_t(r))} \int_0^{L_t(s)} \sin \bigg( \frac{\pi y}{L_t(s)} \bigg) \: dy \nonumber \\
&\lesssim e^{\sqrt{2} z} \sin \bigg( \frac{\pi z}{L_t(r)} \bigg) e^{\sqrt{2}(L_t(s) - L_t(r))}.
\end{align*}
Note that this expression is the expression on the right-hand side of \eqref{innerint}, multiplied by $L_t^2$.  Because this inner integral is squared, our final answer can be obtained by multiplying the right-hand side of \eqref{term2Abound} by $L_t^4$.  That is, we have
\begin{align}\label{term2Ag}
&\int_0^{s - t^{2/3}} \int_0^{L_t(r)} q_r^{\rho, t}(x,z) \bigg( \int_0^{L_t(s)} e^{\sqrt{2} y} q_{r,s}^{\rho, t}(z,y) \: dy \bigg)^2 \: dz \: dr \nonumber \\
&\qquad \qquad \qquad \qquad \lesssim \frac{e^{2 \sqrt{2} L_t(s)}}{L_t^2} \bigg( \frac{s}{t} z(x, L_t) + y(x, L_t) \bigg).
\end{align}
We next consider the case when $s - t^{2/3} < r < s$.   Write
$$I = \int_{s - t^{2/3}}^s \int_0^{L_t(r)} q^{\rho, t}_r(x,z) \bigg( \int_0^{L_t(s)} e^{\sqrt{2} y} q^{\rho, t}_{r,s}(z,y) \: dy \bigg)^2 \: dz$$
Reasoning as in \eqref{term2Bprelim}, we get
\begin{align}\label{prelimg}
I &\lesssim \frac{e^{\sqrt{2} L_t(s)}}{L_t^2} z(x, L_t) \int_{s - t^{2/3}}^s \frac{1}{L_t(r)^2} \nonumber \\
&\qquad \qquad \times \int_0^{L_t(r)} e^{\sqrt{2} z} \sin \bigg( \frac{\pi z}{L_t(r)} \bigg)\bigg( \int_0^{L_t(s)} w_{\tau_t(r,s)} \bigg( \frac{z}{L_t(r)}, \frac{y}{L_t(s)} \bigg) \: dy \bigg)^2 \: dz \: dr.
\end{align}
We now make the substitutions $v = L_t(r) - z$ and $q = 1 - y/L_t(s)$ to rewrite the integral on the last line of \eqref{prelimg} as
$$e^{\sqrt{2} L_t(r)} L_t(s)^2 \int_0^{L_t(r)} e^{-\sqrt{2} v} \sin \bigg( \frac{\pi (L_t(r) - v)}{L_t(r)} \bigg) \bigg( \int_0^1 w_{\tau_t(r,s)} \bigg( 1 - \frac{v}{L_t(r)}, 1 - q \bigg) \: dq \bigg)^2 \: dv.$$
We now use the bound $\sin(\pi(L_t(r) - v)/L_t(r)) \leq \pi v/L_t(s)$ along with the identity $w_s(x,y) = w_s(1-x, 1-y)$ and the fact that $e^{\sqrt{2} L_t(r)} \asymp e^{\sqrt{2} L_t(s)}$ when $s - t^{2/3} < r < s$ to get
\begin{align*}
I \lesssim \frac{e^{2 \sqrt{2} L_t(s)}}{L_t} z(x, L_t) \int_{s - t^{2/3}}^s \frac{1}{L_t(r)^2} \int_0^{L_t(r)} e^{-\sqrt{2} v} v \bigg( \int_0^1 w_{\tau_t(r,s)} \bigg( \frac{v}{L_t(r)}, q \bigg) \: dq \bigg)^2 \: dv \: dr.
\end{align*}
Note that $\int_0^1 w_s(x,y) \: dy$ is bounded above by the probability that Brownian motion, started at $x$, does not hit the origin before time $s$.  This probability is bounded above by a constant multiple of $s^{-1/2} x \wedge 1$.  It follows that
$$I \lesssim \frac{e^{2 \sqrt{2} L_t(s)}}{L_t} z(x, L_t) \int_{s - t^{2/3}}^s \frac{1}{L_t(r)^2} \int_0^{L_t(r)} e^{-\sqrt{2} v} v \bigg( \frac{v^2}{L_t(r)^2 \tau_t(r,s)} \wedge 1 \bigg) \: dv \: dr.$$
Next, we make the substitution $u = \tau_t(r,s)$, so that $du/dr = - L_t(r)^{-2}$.  Then recalling that $\tau_t(s - t^{2/3}, s) \lesssim 1$ and noting that $L_t(r) \asymp L_t$, we get
\begin{align*}
I &\lesssim \frac{e^{2 \sqrt{2} L_t(s)}}{L_t} z(x, L_t) \int_0^{\tau_t(s - t^{2/3}, s)} \int_0^{\infty} e^{-\sqrt{2} v} v \bigg( \frac{v^2}{L_t^2 u} \wedge 1 \bigg) \: dv \: du \\
&= \frac{e^{2 \sqrt{2} L_t(s)}}{L_t^3} z(x, L_t) \int_0^{\infty} e^{-\sqrt{2} v} v \int_0^{\tau_t(s - t^{2/3}, s)} \bigg( \frac{v^2}{u} \wedge L_t^2 \bigg) \: dr \: du \\
&= \frac{e^{2 \sqrt{2} L_t(s)}}{L_t^3} z(x, L_t) \int_0^{\infty} e^{-\sqrt{2} v} v \bigg( \int_0^{(v/L_t)^2} L_t^2 \: du + \int_{(v/L_t)^2}^{\tau_t(s - t^{2/3}, s)} \frac{v^2}{u} \: du \bigg) \: dv \\
&\lesssim \frac{e^{2 \sqrt{2} L_t(s)}}{L_t^3} z(x, L_t) \int_0^{\infty} e^{-\sqrt{2} v} v \bigg( v^2 + v^2 \log(L_t/v) \bigg) \: dv \\
&\lesssim \frac{e^{2 \sqrt{2} L_t(s)} \log L_t}{L_t^3} z(x, L_t).
\end{align*}
Combining this bound with \eqref{varg2}, \eqref{term1g}, and \eqref{term2Ag} gives \eqref{mom4}.
\end{proof}

\paragraph{Acknowledgments.}
This research was carried out in part during the workshop \emph{Branching systems, reaction-diffusion equations and population models}, held in Montreal in May, 2022.  We thank the organizers of this conference. In particular, BM acknowledges partial support by funding from the Simons Foundation and the Centre de Recherches Mathématiques, through the Simons-CRM scholar-in-residence program.


\begin{thebibliography}{99}
\bibitem{ABBS} E. A\"id\'ekon, J. Berestycki, \'E. Brunet, and Z. Shi (2013).  Branching Brownian motion seen from its tip.  {\it Probab. Theory Related Fields} {\bf 157}, 405-451.

\bibitem{bbs13}J. Berestycki, N. Berestycki, and J. Schweinsberg (2013).  The genealogy of branching Brownian motion with absorption.  {\it Ann. Probab.} {\bf 41}, 527-618.

\bibitem{bbs14}J. Berestycki, N. Berestycki, and J. Schweinsberg (2014).  Critical branching Brownian motion with absorption: survival probability.  {\it Probab. Theory Related Fields} {\bf 160}, 489-520.

\bibitem{bbs15} J. Berestycki, N. Berestycki, and J. Schweinsberg (2015).  Critical branching Brownian motion with absorption: particle configurations.  {\it Ann. Inst. H. Poincar\'e Probab. Statist.} {\bf 51}, 1215-1250.

\bibitem{bbcm} J. Berestycki, \'E. Brunet, A. Cortines, and B. Mallein (2022). A simple backward construction of branching Brownian motion with large displacement and applications. {\it Ann. Instit. H. Poincar\'e Probab. Statist.} {\bf 58}, no. 4, 2094-2113.

\bibitem{EnCours} J. Berestycki, J. Liu, B. Mallein, and J. Schweinsberg (2024+). Phase transition of the consistent maximal displacement of branching Brownian motion.  ArXiv preprint 2406.04526.

\bibitem{BiK04} J. D. Biggins and A. E. Kyprianou (2004).  Measure change in multitype branching.  {\it Adv. Appl. Probab.} {\bf 36}, 544-581.

\bibitem{cr88} B. Chauvin and A. Rouault (1988).  KPP equation and supercritical branching Brownian motion in the subcritical speed area: application to spatial trees.  {\it Probab. Theory Related Fields} {\bf 80}, 299-314.

\bibitem{cmsm13} P. Collet, S. Mart\'inez, and J. San Mart\'in (2013).  {\it Quasi-stationary distributions:  Markov chains, diffusions, and dynamical systems}.  Springer, Heidelberg.

\bibitem{hh04} R. Hardy and S. C. Harris (2004). A new formulation of the spine approach to branching diffusions. {\it Mathematics Preprint}, no. 0404, University of Bath.

\bibitem{hh07} J. W. Harris and S. C. Harris (2007).  Survival probabilities for branching Brownian motion with absorption.  {\it Electron. Commun. Probab.} {\bf 12}, 81-92.

\bibitem{inw69} N. Ikeda, M. Nagasawa, and S. Watanabe (1969).  Branching Markov processes III.  {\it J. Math. Kyoto Univ.} {\bf 9}, 95-160.

\bibitem{kal02} O. Kallenberg (2002).  {\it Foundations of Modern Probability}, 2nd. edition.  Springer, New York.

\bibitem{kesten} H. Kesten (1978).  Branching Brownian motion with absorption.  {\it Stochastic Process. Appl.} {\bf 7}, 9-47.

\bibitem{liu21} J. Liu (2021).  A Yaglom type asymptotic result for subcritical branching Brownian motion with absorption.  {\it Stochastic Process. Appl.} {\bf 141}, 245-273.

\bibitem{ll} L. Luo (2024). Precise upper deviation estimates for the maximum of a branching random walk. ArXiv preprint 2403.03687.

\bibitem{Lyo97} R. Lyons (1997).  A simple path to Biggins' martingale convergence for branching random walk.  In {\it Classical and Modern Branching Processes, IMA Vol. Math. Appl.} {\bf 84}, 217-221.

\bibitem{madaule} T. Madaule (2017).  Convergence in law for the branching random walk seen from its tip.  {\it J. Theor. Probab.} {\bf 30}, 27-63.

\bibitem{ms22} P. Maillard and J. Schweinsberg (2022).  Yaglom-type limit theorems for branching Brownian motion with absorption.  {\it Ann. H. Lebesgue} {\bf 5}, 921-985.

\bibitem{MV12} S. M\'el\'eard and D. Villemonais (2012).  Quasi-stationary distributions and population processes.  {\it Probab. Surv.} {\bf 9}, 340-410.

\bibitem{Sch92} T. H. Scheike (1992).  A boundary-crossing result for Brownian motion.  {\it J. Appl. Probab.} {\bf 29}, 448-453.
\end{thebibliography}
\end{document}